\documentclass[notitlepage,a4paper,aps,prd,onecolumn,superscriptaddress,nofootinbib,groupedaddress]{revtex4-1}
\usepackage{enumerate}
\usepackage{amsmath}
\usepackage{bm}
\usepackage{epsfig}
\usepackage{amsfonts}
\usepackage{amssymb}
\usepackage[utf8]{inputenc}
\usepackage[T1]{fontenc}
\usepackage{mathtools}
\usepackage[colorlinks=true]{hyperref}
\usepackage[dvipsnames]{xcolor}
\usepackage{graphicx}
\usepackage[normalem]{ulem}
\usepackage{enumerate}

\newcommand{\be}{\begin{equation}}
\newcommand{\ee}{\end{equation}}
\newcommand{\bea}{\begin{eqnarray}}
\newcommand{\eea}{\end{eqnarray}}
\newtheorem{theorem}{Theorem}

\newtheorem{conclusion}[theorem]{Conclusion}

\newtheorem{definition}[theorem]{Definition}

\newtheorem{lemma}[theorem]{Lemma}

\newenvironment{proof}[1][Proof]{\noindent\textbf{#1.} }{\ \rule{0.5em}{0.5em}}

\begin{document}

\title{Four-dimensional $SO(3)$-spherically symmetric Berwald Finsler spaces}
\author{Samira Cheraghchi}
\email{samira.cheraghchi@unitbv.ro}
\affiliation{Faculty of Mathematics and Computer Science, Transilvania University, Iuliu Maniu Str. 50, 500091 Brasov, Romania}
\author{Christian Pfeifer}
\email{christian.pfeifer@zarm.uni-bremen.de}
\affiliation{ZARM, University of Bremen, 28359 Bremen, Germany}
\author{Nicoleta Voicu}
\email{nico.voicu@unitbv.ro}
\affiliation{Faculty of Mathematics and Computer Science, Transilvania University, Iuliu Maniu Str. 50, 500091 Brasov, Romania}

\begin{abstract}
We locally classify all $SO(3)$-invariant 4-dimensional pseudo-Finsler Berwald structures. These are Finslerian geometries which are closest to (spatially, or $SO(3)$)-spherically symmetric pseudo-Riemannian ones - and serve as ansatz to find solutions of Finsler gravity equations which generalize the Einstein equations. We find that there exist six classes of non pseudo-Riemannian (i.e., non-quadratic in the velocities) $SO(3)$ spherically symmetric pseudo-Finsler Berwald functions, which have either: a power law, an exponential law, or a one- or two-variable dependence on the velocities.
\end{abstract}

\maketitle

\tableofcontents


\section{Introduction}
\label{sec:Intro}

Symmetries are essential in the study of the geometry of manifolds as well as in the mathematical description of physical systems. The geometry of manifolds can be classified according to the symmetries they possess, and if differential equations describing a physical system possess symmetries, finding solutions becomes tremendously simplified due to the existence of conservation laws and a reduction of degrees of freedom.

One fundamental type of symmetry is spherical (or partial spherical) symmetry, i.e.\ the invariance of the geometry of a manifold, or of differential equations and their solutions under an appropriate action of one of the $SO(n)$ groups. Partial spherical symmetry denotes the cases for $n < D$, where $D$ is the dimension of the manifold under consideration. Numerous physical systems possess, at least to a very good approximation, spherical or partial spherical symmetry - and the understanding of such systems lays the foundation for the understanding of more complicated ones, with less symmetry.

In gravitational physics, the gravitational interaction is identified with the geometry of the spacetime manifold, which is one way how physics and geometry come together \cite{Wald}. Thus, in order to study symmetric gravitating systems, spacetime manifolds admitting the same symmetries are needed. 

The local classification of spherically symmetric pseudo-Riemannian metrics in any dimension is well known. From the perspective of physics, Lorentzian manifolds, i.e.\ pseudo-Riemannian manifolds with a metric of Lorentzian signature,  are of particular interest, since in dimension $D=4$, physical spacetime manifolds are those Lorentzian manifolds whose metric solves the Einstein equations; the most famous partial spherically symmetric metric in four dimensions, describing spherical black holes, is the Schwarzschild metric \cite{Wald}.

But, the geometry of manifolds - or, in physics the geometry of spacetimes - can be described in numerous more general ways than by pseudo-Riemannian geometry. One could consider for example, affine geometry (and its subcases teleparallel or symmetric teleparallel geometry), i.e.\ manifolds equipped with general affine connections (with curvature and torsion, only torsion, or only curvature). For these, geometries, which are particularly interesting for gravity theories beyond general relativity \cite{Saridakis:2021vue} and effective models of quantum gravity \cite{Addazi:2021xuf}, spherical symmetry has been investigated in some detail \cite{Hohmann:2019fvf}. 

Another geometry of manifolds or spacetimes equipped with an affine connection is pseudo-Finsler Berwald geometry \cite{Berwald1926}. This can be seen as the subclass of Finsler geometry \cite{Finsler,Bao,Miron} which is the closest to pseudo-Riemannian one; its distinctive feature is that it still leads to an affine connection on the manifold in consideration, even though the fundamental building block of the geometry of the manifold is a \textit{Finsler function}, providing a pseudo-norm for tangent vectors  that does not necessarily arise from a pseudo-scalar product (actually, if the latter happens, then the considered pseudo-Finsler geometry reduces to pseudo-Riemannian geometry).

In mathematics, positive definite Finsler geometry is a well studied subject. However, in pseudo-Finsler geometry, i.e.\ Finsler geometry with a Finsler metric that is not necessarily positive definite, a lot of questions are still open. From the mathematical perspective, it is of high interest, since a lot of questions and classifications answered in positive definite Finsler geometry do not carry over to pseudo-Finsler geometry, see \cite{Voicu2018,Fuster:2020upk,Javaloyes:2022hph,Pfeifer:2011tk,Minguzzi:2014fxa,Gomez-Lobo:2016qik,Minguzzi2016,Javaloyes:2018lex,Minculete_2021}. Its most prominent application is the one of Finsler spacetimes in classical and quantum gravitational physics \cite{Tavakol1986,Voicu:2009wi,Pfeifer:2011xi,Lammerzahl:2018lhw,Pfeifer:2019,Hohmann_2019,Hohmann:2019sni,Lobo:2020qoa,Addazi:2021xuf, Kapsabelis:2022plf,Carvalho:2022sdz,Garcia-Parrado:2022ith,Zhu:2022blp,Aazami:2022bib,Javaloyes:2022fmp,Hohmann:2020mgs} as extension of general relativity. 
\\
Passing to Berwald-Finsler spherical symmetry, a first important step was made by Elgendi, \cite{Elgendi:2021-1}, who classified all such functions admitting "maximal" spherical symmetry, understood as invariance under $SO(D)$. Yet, for physics, a most prominent situation is the one of $SO(3)$-spherical symmetry of a $D=4$ dimensional Lorentzian manifold, the physical spacetime. Then, this becomes precisely \textit{spatial} spherical symmetry of spacetime manifolds equipped with a time function - and it is of interest to predict the gravitational field of black holes, dust clouds, ordinary, neutron or boson stars based on a Finslerian extension of general relativity. As we will see below, this assumption enlarges quite substantially the class of admissible Finsler functions, compared to the situation of "maximal" spherical symmetry.

\bigskip

In this article, we will derive the coordinate expressions of all possible classes of 4-dimensional pseudo-Finsler Berwald functions that are invariant under the action of $SO(3)$. In particular, if one imposes Lorentzian signature and interprets the dimension which is not affected by the rotations as \textit{time}, these are interpreted as \textit{spatially spherically symmetric Finsler-Berwald spacetime structures}. This selection is tailored for the search of spatially spherically symmetric solutions of the Finsler gravity equations \cite{Hohmann:2018rpp} and its application to the gravitational field of kinetic gases \cite{Hohmann:2019sni,Hohmann:2020yia}. However, the results we find also apply to any pseudo-Finsler Berwald space of dimension $4$. Moreover, pseudo-Berwald spaces are the starting point to construct so called unicorn Finsler spaces, i.e.\ Finsler spaces with vanishing Landsberg tensor, which are not Berwald \cite{Elgendi:2021-1,ELGENDI2021103918}.
\\
The method we use here is the following. Instead of directly imposing the vanishing the so-called Berwald curvature tensor (that singles out Berwald spaces among Finsler spaces), which leads to a nonlinear, second order PDE system, we use the property of affine connectedness of Berwald spaces. More precisely, we start from the most general form of SO(3)-invariant affine connections on a 4-dimensional manifold, which is known from \cite{Hohmann:2019fvf} and find all compatible pseudo-Finsler functions (which will thus be compulsorily of Berwald type); in other words, we reinterpret our problem as a \textit{pseudo-Finsler metrizability} one for an affine connection, see \cite{Muzsnay2006TheEP}, \cite{Bucataru_2011}. This technique has two major advantages. On the one hand, the resulting PDE system is a linear first order one - i.e., much simpler - and, on the other hand, the consistency conditions of this PDE system offer valuable, direct information on the curvature of the resulting pseudo-Finsler spaces. A similar approach was used by the two latter authors together with M. Hohmann, in clasifying cosmologically symmetric Berwald spacetimes, in \cite{Hohmann:2020mgs}.

The structure of the article is as follows: In Section~\ref{sec:BerwFins} we recall the basic definitions of Finsler geometry and Finsler geometry of Berwald type. In particular, we present a further simplified version of a necessary and sufficient condition for a  pseudo-Finsler space to Berwald in Theorem \ref{thm:berw}. In Section~\ref{sec:SphBerw} we evaluate the Berwald condition for 4-dimensional $SO(3)$-symmetric pseudo-Finsler functions and give their classification in Theorems \ref{thm:main} and \ref{thm:main2}. The proofs of these two theorems are presented in Section~\ref{sec:mainpre}, before we conclude in Section \ref{sec:conc}.
 

\section{Berwald Finsler Geometry}
\label{sec:BerwFins}

 We begin by recalling the basic notions of Finsler
geometry \cite{Bao,Miron}. Since our results apply to manifolds with a Finslerian geometry of arbitrary signature, we will introduce the notion of a pseudo-Finsler space with the minimal requirements we need. In particular, our results hold for Finsler spaces and Finsler spacetimes \cite{Beem,Minguzzi:2014fxa,Javaloyes:2018lex,Hohmann:2021zbt}. Further, we prove a refinement of the first order
partial differential system introduced in \cite{Pfeifer:2019tyy}, which provides a necessary and sufficient condition for a Finsler function to be of Berwald type. 
Finally, we recall how to identify $SO(3)$-spherically symmetric pseudo-Finsler spaces \cite{Pfeifer:2011xi}.

Throughout this article we use the following notations. Indices $a,b,c,....$
run from $0$ to $3$. All considered manifolds $M$ are $4$-dimensional and their tangent bundles $TM$ are equipped with manifold induced coordinates, i.e.\ a given coordinate
chart $(U,(x^{a}))$ on $M$ induces a coordinate chart $(TU,(x^{a},\dot x^{a}))$ on $TM$ in by the rule: $Z \in TU$ is labeled by $T_xM \ni Z=\dot x^a {\partial_a}$. The local coordinate basis of $T_{(x,\dot x)}M$
is given by $\{\partial_a = \tfrac{\partial}{\partial x^a}, \dot \partial_a
= \tfrac{\partial}{\partial \dot  x^a} \}$ and the local coordinate basis of $%
T^*_{(x,\dot x)}M$ is $\{dx^a,d\dot x^a \}$. If there is no risk of confusion, we will skip the indices of the coordinates, i.e., refer to $(x^{a},\dot x^{a})$ briefly as~$(x,\dot x)$.


\subsection{Finsler geometry}
\label{ssec:Fins} This section briefly reviews the notion of pseudo-Finsler space $(M,L)$ and its canonically associated geometric objects.

A \textit{conic subbundle} of the tangent bundle $(TM, \pi, M)$ is an open subset $\mathcal{A} \subset TM \setminus 0$ with $\pi(\mathcal{A})=M$, which is stable under positive rescaling of vectors, i.e., for all $(x,\dot{x}) \in \mathcal{A}$ and all $\lambda \in (0,\infty)$, one must get: $(x,\lambda \dot{x}) \in \mathcal{A}$.

\begin{definition}, \cite{Bejancu}:
Let $M$ be a manifold and $\mathcal{A} \subset TM \setminus {0} $, a conic subbundle. A pseudo-Finsler sructure on $M$ is a smooth function $L:\mathcal{A} \to \mathbb{R}$ such that:

\begin{itemize}
\item $L(x,\lambda \dot x) = \lambda^2 L(x,\dot x)$ for all $\lambda > 0$;

\item at all $(x,\dot{x}) \in \mathcal{A}$ and in any local chart around $(x,\dot{x})$, the matrix
\begin{align}
g_{ab}(x,\dot x) := \frac{1}{2}\dot{\partial}_a\dot{\partial}_b L(x,\dot x)
\end{align}
is non-degenerate.
\end{itemize}
\end{definition}

Any pseudo-Finsler metric can be continuously prolonged as $0$ at $\dot{x}=0$. The mapping $g:\mathcal{A}\rightarrow T^{0}_{2}M, (x,\dot{x})\mapsto g_{(x,\dot{x})}=g_{ab}dx^{a} \otimes dx^{b}$, is called the $L$-metric.

The classical $1$-homogeneous Finsler function is defined as $F(x,\dot x) =
\epsilon \sqrt{|L(x,\dot x)|}$ with $\epsilon = \text{sign}(L)$ and in turn, it defines the arc length for curves $x(\tau)$ on $(M,L)$ as: 
\begin{align}  \label{eq:flength}
S[x] = \int F(x,\dot x) d\tau.
\end{align}

Pseudo-Finsler spaces include classical positive definite Finsler spaces, obtained when 
$\mathcal{A} = TM\setminus \{0\}$ and $g$ is positive definite. Different
versions of pseudo-Finsler spaces are obtained when appropriate signature
conditions for $g$ are added \cite{Beem,Minguzzi:2014fxa,Javaloyes:2018lex,Hohmann:2019sni}. For example,
demanding that, for all $x\in M$, there exists a connected component $\mathcal{T}_x$ of the fiber $\mathcal{A}_{x} = \mathcal{A} \cap T_xM$ such that the signature of $g$ is Lorentzian $(+,-,-,-)$ and $L>0$ on $\mathcal{T}_x$, leads to \textit{Finsler spacetimes} as considered in \cite{Hohmann:2021zbt}.

The geometry of pseudo-Finsler spaces is encoded in canonical tensor fields on $\mathcal{A}\subset TM \setminus {0}$, which are constructed from derivatives of $L$. We
briefly list here the local coordinate expressions of those which will be relevant for further considerations.

\begin{itemize}
\item The fundamental building block for the geometry of a
pseudo-Finsler space $(M,L)$ are the \emph{geodesic spray coefficients} 
\begin{align}  \label{eq:geodspray}
G^a = \frac{1}{4}g^{ab}\left(\dot x^c \partial_c \dot \partial_b L -
\partial_b L\right)\,.
\end{align}
They appear naturally in the geodesic equation for arc length parametrized
curves $c:[a,b]\rightarrow M, s\mapsto x(s)$, 
\begin{align}  \label{eq:geodeq}
\frac{d^2 x^a}{ds^2} + 2 G^a\left(x,\tfrac{dx}{ds}\right) = 0\,,
\end{align}
derived from extremizing the length functional \eqref{eq:flength}.

\item The \emph{Cartan nonlinear connection coefficients} are derived from
the geodesic spray as 
\begin{align}  \label{eq:CNLin}
N^a{}_b = \dot\partial_b G^a\,.
\end{align}

\item The \emph{local adapted bases} to the above connection, of the (co)tangent bundles
$T\mathcal{A}$, $T^{*}\mathcal{A}$ are, respectively: 
\begin{align}
\{\delta_a =
\partial_a - N^b{}_a \dot \partial_b,\dot \partial_a\}\,,\quad \{dx^a,\delta\dot x^a = d\dot x^a + N^a{}_b dx^b\}\,.
\end{align}
An important fact is that the Finsler Lagrangian $L$ is horizontally
constant; in coordinates, this reads:
\begin{align}  \label{eq:dL=0}
\delta_a L = 0\, \quad a=0,...,3.
\end{align}

\item In addition to the Cartan nonlinear connection, it is possible to construct several linear connections on $\mathcal{A} \subset TM\setminus \{0\}$. For us, the \emph{Berwald linear
connection} is of importance; it is defined by its covariant derivative of
the adapted basis elements as: 
\begin{align}  \label{eq:BerwLin}
\nabla_{\delta_a} \delta_b = \dot \partial_a N^c{}_b \delta_c,\quad
\nabla_{\delta_a} \dot \partial_b = \dot \partial_a N^c{}_b
\dot\partial_c,\quad \nabla_{\dot\partial_a} \delta_b = 0,\quad
\nabla_{\dot\partial_a} \dot \partial_b = 0\,.
\end{align}
It follows from \eqref{eq:dL=0} that $\nabla_{\delta_a}L = 0$.
\end{itemize}

With these notions, we can define Berwald manifolds and give a necessary and
sufficient condition, in terms of a first order partial differential
system, to identify them. Afterwards, we recall the notion of $SO(3)$-spherical symmetry on a pseudo-Finsler space.


\subsection{Berwald geometry}
Berwald-Finsler spaces are considered as the pseudo-Finsler spaces which are closest to pseudo-Riemannian ones, \cite{Berwald1926}. They can be characterized in various ways, \cite{Szilasi2011,Pfeifer:2019tyy}. Their
main feature is that the Berwald linear connection \eqref{eq:BerwLin} descends into a well defined affine connection, with coefficients $\Gamma^c{}_{ab}(x) $, on the base manifold $M$; equivalently, in any local chart, the Cartan nonlinear connection coefficients \eqref{eq:CNLin} are linear in $%
\dot x$, or the geodesic spray coefficients \eqref{eq:geodspray} are quadratic in $%
\dot x$: 
\begin{align}\label{eq:BerwAffCon}
	\dot \partial_a N^c{}_b = \Gamma^c{}_{ab}(x) \quad\Leftrightarrow\quad
	N^c{}_b = \Gamma^c{}_{ab}(x)\dot x^a \quad\Leftrightarrow\quad 2 G^c =
	\Gamma^c{}_{ab}(x)\dot x^a\dot x^b\,.
\end{align}

Surprisingly, though $G^a$ involve second order derivatives of the Finsler function $L$, one can identify Finsler functions of Berwald type (those which define a Berwald Finsler manifold) with the help of a \textit{first order} partial
differential system. The idea was first brought up for special Berwald manifolds in \cite{Tavakol1986} and then generalized in \cite{Pfeifer:2019tyy}. 
Here, we prove a further simplified version of this Berwald condition.
The strategy used below actually reduces the problem of deciding whether a pseudo-Finsler function is of Berwald type, to a problem of \textit{pseudo-Finsler metrizabillity} of a (particular class of)second order Euler Lagrange PDE systems, as discussed in \cite{Muzsnay2006TheEP,Bucataru_2011}. Actually, the statement below is also a refinement of Proposition 3 (characterizing Landsberg metrizability of sprays) of the paper by Muzsnay, \cite{Muzsnay2006TheEP}.

\begin{theorem}[Berwald Condition]\label{thm:berw}
	Let $(M,L)$ be a pseudo-Finsler space, then the following statements are
	equivalent:
	
	\begin{enumerate}
		\item $L$ is of Berwald type.
		
		\item There exists a symmetric (torsion-free) connection on $M,$
		with coefficients $\Gamma^c{}_{ab}=\Gamma^c{}_{ab}(x)$ such that, with respect to the induced horizontal derivative 
		$\delta_a = \partial_a - \Gamma^c{}_{ab}(x)\dot x^b\dot\partial_c$: 
		\begin{equation}
			\delta_aL =0.
			\label{Berwald_cond}
		\end{equation}
	\end{enumerate}
	
	Moreover, if $L$ is of Berwald type, then the connection with the property 2. is uniquely determined - and it is the canonical (Berwald) connection of $(M,L)$.
	
\end{theorem}

\begin{proof}
	
	\begin{itemize}
		\item $1 \Rightarrow 2$: If $L$ is of Berwald type, then the statement is satisfied by its Berwald linear connection on~$\mathcal{A}$, whose connection coefficients are defined by \eqref{eq:BerwLin} - and which, in this case, satisfy, in any local chart, \eqref{eq:BerwAffCon}.		
		
		\item $2 \Rightarrow 1$: Assume \eqref{Berwald_cond} holds. Also, we know
		that for any Finsler Lagrangian and its canonical nonlinear connection  $N$, the equation $\delta_a L = \partial_a L-N^b{}_a \dot
		\partial_b L = 0$ holds. Subtracting both equations yields 
		\begin{align}  \label{eq:BerwCondProof1}
			\left(\Gamma^c{}_{ab}(x)\dot x^b - N^c{}_a\right)\dot\partial_c L = 0\,.
		\end{align}
		Differentiating with respect to $\dot{x}^{d},$ we then find 
		\begin{align}
			\left(\Gamma^c{}_{ad}(x) - \dot\partial_d N^c{}_a\right)\dot\partial_c
			L+\left(\Gamma^c{}_{ab}(x)\dot x^b - N^c{}_a\right)2 g_{cd} = 0\,.
		\end{align}
		Contracting this equation again with $\dot x^a$ gives 
		\begin{align}
			\left(\Gamma^c{}_{ad}(x)\dot x^a - N^c{}_d\right)\dot\partial_c
			L+\left(\Gamma^c{}_{ab}(x)\dot x^b\dot x^a - 2 G^c\right)2 g_{cd} = 0\,.
		\end{align}
		Employing \eqref{eq:BerwCondProof1}, we are left with 
		\begin{align}
			\left(\Gamma^c{}_{ab}(x)\dot x^b\dot x^a - 2 G^c\right)2 g_{cd} = 0\,.
		\end{align}
		Since $g_{cd}$ are the components of a non-degenerate metric, we get
		\begin{align}  \label{eq:BerwCondProof2}
			\Gamma^c{}_{ab}(x)\dot x^b\dot x^a = 2 G^c\,,
		\end{align}
		in particular, $G^c$ are quadratic in $\dot{x}$ and thus $L$ must be of Berwald type.
		
	\end{itemize}
	Finally, assuming that $(M,L)$ is Bwrald, the uniqueness statement follows immediately by differentiating \eqref{eq:BerwCondProof2} twice with respect to $\dot{x}^a,\dot{x}^b$.	
\end{proof}

\vspace{12pt}

Thus, finding a Berwald structure $L$ on $M$ reduces to solving the equations 
\begin{align}  \label{Berwald_cond2}
	\delta_a L= \partial_a L- \Gamma^c{}_{ab}(x)\dot x^b\dot\partial_c L=0\,,
\end{align}
for a given torsion-free affine connection on $M$.

\bigskip

Thus, in Berwald-Finsler geometry, the Berwald connection is, similarly to the Levi-Civita connection in Riemannian geometry, the unique affine torsionless connection on $M$, which is metric (with respect to $L$). 

\noindent Actually, for $TM\setminus{\{0\}}$-smooth positive definite Finsler manifolds, it is known by Szabo's Theorem \cite{Szabo} that if the given connection is the canonical connection of a Berwald metric, then it is also the Levi-Civita connection for some Riemannian metric. This is in general not true for Berwald manifolds of arbitrary signature \cite{Fuster:2020upk}.


\subsection{Spatial Spherical symmetry}

We are recalling here the results derived in all detail in \cite{Pfeifer:2011xi}.
A manifold induced symmetry of a pseudo-Finsler space $(M,L)$ is a
diffeomorphism $\psi:M\to M$, whose natural lift $\Psi = T\psi$ to $TM$ leaves the
Finsler Lagrangian invariant: 
\begin{align}  \label{eq:symm}
L \circ \Psi = L\,.
\end{align}
Infinitesimally, if $X = \xi^a(x)\partial_a$ is the generator of a 1-parameter group of diffeomorphisms $\psi_{\varepsilon}$ on $M$, then the natural lifts $\Psi_{\varepsilon}$ are generated by the complete lift $X^C=\xi^a(x)\partial_a + \dot x^b \partial_b \xi^a(x) \dot\partial_b$ of $X$ to $TM$. The symmetry condition is translated into
\begin{align}
X^C(L)=0\,.
\end{align}

\bigskip

For the rest of the paper, we will consider manifolds $M$ such that they admit a chart $U\subset M$, which can be covered by spherical coordinates $(x^a):=(t,r,\theta,\phi)$ - and we will work in such a chart\footnote{We will not discuss the various topologies of manifolds that can support $SO(3)$ spherical symmetry; for such a discussion, we refer, e.g., to the recent paper by Krupka and Brajercik, \cite{Krupka-Brajercik}.}. 

Inspired by the application to spacetime physics, we will call $t$ the time coordinate and $(r,\theta,\phi)$, spatial spherical coordinates, though we do not make any assumption on the signature of the $L$-metric.
Accordingly, we will call \textit{spatial spherical symmetry}, the invariance of $L$ under the action of $SO(3)$ involving only the spatial coordinates (i.e., the action of $SO(3)$ on the spatial sheets $t=const.$).

\bigskip

Thus, spatial spherical symmetry is defined by three vector fields, which generate the $\mathfrak{so}(3)$ Lie algebra. In the coordinates $(t,r,\theta,\phi,\dot{t},\dot{r},\dot{\theta},\dot{\phi})$ induced by the local spherical coordinates $(x^a):=(t,r,\theta,\phi)$ on $M$, their
complete lifts to the tangent bundle are given by: 
\begin{align}
\begin{split}
&X^C_1=\sin\phi\partial_\theta+\cot\theta\cos\phi\partial_\phi+\dot{\phi}%
\cos\phi\dot{\partial}_\theta-\bigg(\dot\theta\frac{\cos\phi}{\sin^2\theta}%
+\dot\phi\cot\theta\sin\phi\bigg)\dot{\partial}_\phi\,, \\
&X^C_2=-\cos\phi\partial_\theta+\cot\theta\sin\phi\partial_\phi+\dot{\phi}%
\sin\phi\dot{\partial}_\theta-\bigg(\dot\theta\frac{\sin\phi}{\sin^2\theta}%
+\dot\phi\cot\theta\cos\phi\bigg)\dot{\partial}_\phi\,, \\
&X^C_3=\partial_\phi\,.
\end{split}%
\end{align}
Applying the symmetry condition $X^C_I(L)=0, I=1,2,3$, yields that the most general spatially spherically symmetric Finsler function on $U$ is of the form
\begin{align}  \label{eq:sphL}
L(t,r,\theta,\phi,\dot t, \dot r,\dot \theta, \dot \phi) = L(t,r,\dot t,\dot
r,w), \quad \text{ with }\quad w^2 = \dot
\theta^2 + \sin^2\theta \dot \phi^2\,.
\end{align}

For spherically symmetric functions as above, the partial derivatives with respect to $\theta, \dot{\theta}$ and $\dot{\phi}$ can be expressed in terms of $w$-derivatives as
\begin{align}\label{eq:tow}
	\partial_\theta = \tfrac{\dot\phi^2 \sin\theta\cos\theta}{w}\partial_w\,, \quad
	\dot \partial_{\theta} = \tfrac{\dot \theta}{w}\partial_w\,,\quad 	 
	\dot \partial_\phi = \tfrac{\dot \phi \sin^2\theta}{w}\partial_w\,.  \quad	
\end{align}

\bigskip

Accordingly, we obtain:
\begin{align} \label{def:partial_w}
\dot \theta \dot \partial_\theta + \dot \phi \dot\partial_\phi = w\partial_w\,.
\end{align}

The above relation can actually be regarded as the \textit{definition} of a local vector field $\partial_{w}$ - that may act also on (not necessarily spherically symmetric) functions.

This insight on the simplification of the dependence of $L$ implied by spherical symmetry will allow us to solve the Berwald condition.


\section{The Berwald condition in spherical symmetry}\label{sec:SphBerw}

To write down the Berwald condition \eqref{Berwald_cond2}
for the most general spherically symmetric Finsler Lagrangian of Berwald
type, we need the most general spherically symmetric torsion free affine
connection coefficients on $M$ as input. We will perform all derivations in local spherical coordinates $(t,r,\theta,\phi)$ on a fixed chart.

In \cite{Hohmann:2019fvf}, these connection coefficients have been found. In
general, they are parametrized by $20$ free functions $k_{I}=k_{I}(t,r)$ of $%
t$ and $r$, in $4$ dimensions. Imposing the vanishing torsion condition,
i.e. the symmetry in the lower indices of the affine connection
coefficients, one is left with $12$ free functions, which appear in the
nonvanishing affine connection components, as follows:
\begin{align}
\Gamma _{tt}^{t}& =k_{1}(t,r), & \Gamma _{tr}^{t}& =k_{2}(t,r), & \Gamma
_{rr}^{t}& =k_{3}(t,r), & \Gamma _{tt}^{r}& =k_{4}(t,r), \label{eq:ac1}\\
\Gamma _{rr}^{r}& =k_{5}(t,r), & \Gamma _{tr}^{r}& =k_{6}(t,r), & \Gamma
_{\theta \theta }^{t}& =\frac{\Gamma _{\phi \phi }^{t}}{\sin ^{2}\theta }%
=k_{7}(t,r), & \Gamma _{\phi t}^{\phi }& =\Gamma _{\theta t}^{\theta
}=k_{8}(t,r), \label{eq:ac2}\\
\Gamma _{\phi r}^{\phi }& =\Gamma _{\theta r}^{\theta }=k_{9}(t,r), & \Gamma
_{\theta \theta }^{r}& =\frac{\Gamma _{\phi \phi }^{r}}{\sin ^{2}\theta }%
=k_{10}(t,r), & \sin \theta \Gamma _{t\theta }^{\phi }& =-\frac{\Gamma
_{\phi t}^{\theta }}{\sin \theta }=k_{11}(t,r), & & \label{eq:ac3}\\
\sin \theta \Gamma _{r\theta }^{\phi } &=-\frac{\Gamma _{r\phi }^{\theta }}{\sin \theta }=k_{12}(t,r), & 
\Gamma _{\theta \phi }^{\phi } &=\Gamma _{\phi\theta }^{\phi }=\cot {\theta }, 
& \Gamma _{\phi \phi }^{\theta } &=-\sin
\theta \cos \theta \,. & &\label{eq:ac4}
\end{align}

Further, for the Berwald condition, the explicit expressions of the Cartan nonlinear connection
coefficients $N^{a}{}_{b}=\Gamma ^{a}{}_{bc}(x)\dot{x}^{c}
$ are needed: 
\begin{align}
N_{~t}^{t}& =k_{1}\dot{t}+k_{2}\dot{r}, & N_{~t}^{r}& =k_{4}\dot{t}+k_{6}%
\dot{r}, & N_{~t}^{\theta }& =k_{8}\dot{\theta}-k_{11}\dot{\phi}\sin \theta ,
& N_{~t}^{\phi }& =\frac{k_{11}}{\sin {\theta }}\dot{\theta}+k_{8}\dot{\phi},
\label{G's} \\
N_{~r}^{t}& =k_{2}\dot{t}+k_{3}\dot{r}, & N_{~r}^{r}& =k_{6}\dot{t}+k_{5}%
\dot{r}, & N_{~r}^{\theta }& =k_{9}\dot{\theta}-k_{12}\dot{\phi}\sin \theta ,
& N_{~r}^{\phi }& =\frac{k_{12}}{\sin {\theta }}\dot{\theta}+k_{9}\dot{\phi},
\\
N_{~\theta }^{t}& =k_{7}\dot{\theta}, & N_{~\theta }^{r}& =k_{10}\dot{\theta}%
, & N_{~\theta }^{\theta }& =k_{8}\dot{t}+k_{9}\dot{r}, & N_{~\theta }^{\phi
}& =\frac{k_{11}}{\sin {\theta }}\dot{t}+\frac{k_{12}}{\sin {\theta }}\dot{r}%
+\dot{\phi}\cot \theta , \\
N_{~\phi }^{t}& =k_{7}\dot{\phi}\sin ^{2}\theta , & N_{~\phi }^{r}& =k_{10}%
\dot{\phi}\sin ^{2}{\theta }, & N_{~\phi }^{\theta }& =(-k_{11}\dot{t}-k_{12}%
\dot{r}-\dot{\phi}\cos {\theta })\sin {\theta }, & N_{~\phi }^{\phi }& =k_{8}%
\dot{t}+k_{9}\dot{r}+\dot{\theta}\cot {\theta }\,.
\end{align}

\bigskip

A first simplification of the Berwald condition \eqref{Berwald_cond}, in spherical symmetry, is given by the following Lemma.

\begin{lemma}
The (sub-)system consisting of the equations $\delta_\theta L = 0 = \delta_\phi L$ is equivalent to:

\begin{align} 
\delta_w L &:= \left(w  k_7 \dot \partial_t  +w  k_{10} \dot \partial_r + (k_8 \dot t + k_9 \dot r)\partial_w \right)L = 0, \label{eq:delta_w_L} \\
 k_{11} &= k_{12} = 0\,. \label{eq:k11k12vanish}
\end{align}
\end{lemma}

\begin{proof}
The system $\delta_\theta L = 0 = \delta_\phi L$, is equivalent to
\begin{align} \label{eq:delta_w_def}
- \frac{w}{\dot \theta}\left(\delta_\theta + \tfrac{\dot \theta}{\sin^2\theta \dot \phi} \delta_\phi \right)L = 0\,,\quad \textrm{and}\quad  - \frac{w}{\dot \theta}\left(\delta_\theta - \tfrac{\dot \theta}{\sin^2\theta \dot \phi} \delta_\phi \right)L = 0\,,
\end{align}
The first of these equations, written in coordinates, becomes \eqref{eq:delta_w_L}, whereas the second one implies that either $\partial_w L = 0$ (which would lead to a degenerate metric tensor $g$ and thus cannot be a valid solution), or necessarily $k_{11} =k_{12}=0\,.$
\end{proof}

\bigskip

Therefore, in the following, we will always consider $k_{11} =k_{12}=0$. Moreover, it turns out it is convenient to introduce the locally defined vector field:
\begin{align} \label{def_delta_w}
\delta_w := w  k_7 \dot \partial_t  +w  k_{10} \dot \partial_r + (k_8 \dot t + k_9 \dot r)\partial_w \,,
\end{align}
where $\partial_w$ is defined by \eqref{def:partial_w}.
Note: The symbol $\delta_w$ is chosen just for the uniformity of writing; of course, $\delta_w$ is not an element of the adapted basis to the connection $N$. 
To be more precise, it is related, e.g., to $\delta_{\theta}$ by:
\begin{align} \label{eq:rel_theta_w}
\delta_{\theta}=\left(\partial_{\theta}-\dot{\phi}\cot\phi\partial_{\dot{\phi}} \right)-\frac{\dot{\theta}}{w}\delta_w.
\end{align}

Further, expressing also the remaining equations \eqref{Berwald_cond2} and adding the homogeneity condition,  we end up with the following four equations:
\begin{align}
	\delta_t L &= \partial_t L - (k_1 \dot t + k_2 \dot r) \dot \partial_t L - (k_4 \dot t + k_6 \dot r) \dot \partial_r L - k_8 w \partial_w L = 0 \label{eq:deltt}\,,\\
	\delta_r L &= \partial_r L - (k_2 \dot t + k_3 \dot r) \dot \partial_t L - (k_6 \dot t + k_5 \dot r) \dot \partial_r L - k_9 w \partial_w L = 0 \label{eq:deltr}\,,\\
	\delta_w L &=w  k_7 \dot \partial_t L +w  k_{10} \dot \partial_r L + (k_8 \dot t + k_9 \dot r)\partial_w L = 0 \label{eq:deltaw}\,,\\
	2 L &= \dot t \dot \partial_t L + \dot r \dot \partial_t L + w \partial_w L \,. \label{eq:homL}
\end{align}

Solving these equations leads to the two main theorems of our work, which we will prove in the following sections. Before stating these results, some remarks are necessary:
\begin{itemize}
	\item Since we are looking for a local coordinate characterization of Berwald-Finsler functions with $SO(3)$-sperical symmetry, the relevant situation is $M:=U,$ where $U\subset \mathbb{R}^{4}\setminus \{0\} $ is a chart domain.
	\item The system \eqref{eq:deltt}-\eqref{eq:homL} is an overdetermined PDE system. Its consistency conditions will be expressed (as we will see in the next sections - and as expected, taking into account previous works on Finsler metrizability of sprays, e.g., \cite{Muzsnay2006TheEP, Bucataru_2011}) in terms of the Lie brackets of the vector fields $\delta_t,\delta_r,\delta_w$. 
	
	\item Equation \eqref{eq:deltaw} does not contain derivatives w.r.t. the coordinates $t$ and $r$, which is why we will solve it first. But, if $k_7,...,k_{10}$ are all zero, it becomes an identity - and this case needs a separate treatment. Moreover, if \eqref{eq:deltaw} is \textit{not} an identity, at least one of the coefficients $k_7,k_{10}$ must be nonzero (or else, we would get $\partial_{w}L=0$, which is not a valid solution).
	
	\item Assuming, for instance, that $k_{10} \neq 0$, the following quantities make sense in the given chart:
	\begin{align}\label{A-F_defdef}
	\begin{split}
	a &= \tfrac{k_7}{k_{10}},\quad b = \tfrac{k_8}{k_{10}},\quad c = \tfrac{k_9 k_{10} - k_7 k_8}{k_{10}^2}\\
	A &:=b\left( aa_{1}+a_{2}\right) +\left( ab+c\right) \left(aa_{3}+a_{4}\right) -a_{5}\left( 2ab+c\right)\,,   \\
	B &:=a\left( aa_{3}+a_{4}\right) -\left( aa_{1}+a_{2}\right)\,, \\
	C &:=\left( ab+c\right) a_{3}+b\left( aa_{3}+a_{4}\right) +b (a_{1}-2a_{5})\,, \\
	D &:= aa_{3}-a_{1}+a_{5}\,, \\
	E &:=ba_{3}\,, \\ 
	F &:=aa_{3}-a_{1}\,,\\
	M &:= 2 (k_1 - k_4 a)\,, \quad \tilde M := M - 2k_8\,,\\
	N &:= 2 (k_2 - k_6 a)\,, \quad \tilde N := N - 2k_9\,. 
	\end{split}
	\end{align}
	as well as:
	\begin{align}
	u = \dot t - a \dot r,\quad v = c \dot r^2 - 2 b \dot t\dot r - w^2\,.
	\end{align}
	
\end{itemize}

\begin{theorem}[The case $\delta_w\neq0$]:\label{thm:main}
	Let $(M,L = L(t,r,\dot t, \dot r, w))$ be an $SO(3)$-spatially symmetric 4-dimensional pseudo-Finsler space and let $\Gamma$ be a spherically symmetric affine connection on $M$, with connection coefficients as in \eqref{eq:ac1} -\eqref{eq:ac4} and curvature coefficients $a_i$, \eqref{a_i}. Assume $k_{10} \neq 0$. Then: \\
	I. If $(M,L)$ is of Berwald type and non-pseudo-Riemannian, with  canonical connection $\Gamma$, then $\Gamma$ must satisfy:
	\begin{enumerate}[a)]
		\item $k_{11}=k_{12}=0$.
		\item $[\delta_t, \delta_w]$ and $[\delta_r, \delta_w]$ are both proportional to $\delta_w$ and $[\delta_t,[\delta_t,\delta_r]], [\delta_r,[\delta_t,\delta_r]]$ are both proportional to $[\delta_t,\delta_r]$.
		\item $A=B=C=0$.
	\end{enumerate}	
	II. Assume $\Gamma$ satisfies all the conditions above. Then, $L$ is of Berwald type, admitting $\Gamma$ as its canonical connection, in precisely one of the following situations:
		\begin{enumerate}[1.]
			\item The vector fields $[\delta_t, \delta_r]$ and $\delta_w$ are not proportional (that is, $D,E,F$ are not all zero) and:
			\begin{enumerate}[i.)]
				\item $D\neq 0$ (which leads to $\lambda = F/D = const.$). In this case, $L$ is given by:
				\begin{align}\label{eq:1-L2}
					L = \vartheta(t,r) u^{2-2 \lambda}\left(v + \rho u\right)^\lambda\,, 
				\end{align}
				with $\vartheta=\vartheta(t,r)$ and $\rho=\rho(t,r)$ satisfying
				\begin{align}
					\rho = \frac{D}{E}\,, \quad \vartheta(t,r) &= e^{\int (M - \lambda \tilde M)dt} =  e^{\int (N - \lambda \tilde N)dr}\,.
				\end{align}
				\item $D= 0$, $E\neq 0$. In this case:
				\begin{align}\label{eq:1-L3}
					L = \varphi(t,r) u^2 e^{\frac{v}{u^2}\mu}\,, 
				\end{align}
				with $\mu = F/E$ and $\varphi$ determined by
 				\begin{align}\label{eq:phithm}
					\varphi = e^{\int dt (M + 2 k_4 b \mu)} = e^{\int dr (N + 2 k_6 b \mu)}\,.
				\end{align}
			\end{enumerate}
		\item $[\delta_t, \delta_r] \sim \delta_w$ (that is $D=E=F=0$) and:
		\begin{enumerate}[i.)]
			\item $b=c=0$. In this case, $L$ is of the form
			\begin{align}\label{eq:1-L4}
				L = u^2 \Xi = -w^2 \xi(q)\,,\quad q= - \frac{w^2}{u^2}e^{-f(t,r)}
			\end{align}
			with $f$ determined by integration along an arbitrary curve $C$ connecting some initial point $(t_0,r_0)$ to $(t,r)$ of
			\begin{align}
				f = \int_C (M dt + N dr)\,.
			\end{align}
			\item $b,c$ - not both zero, $[\delta_t, \delta_r]=0$. In this case, there exists a coordinate change $(t,r)\mapsto (\tilde t,\tilde{r})$ such that, in the new coordinates, $L=u^2\Xi(z)$ is independent of $(\tilde t,\tilde{r})$:
			\begin{align}\label{eq:1-L5}
				L = u(\dot {\tilde  t}, \dot {\tilde r}, w)^2 \Xi(z(\dot {\tilde  t}, \dot {\tilde r}, w))\,, \quad z(\dot {\tilde  t}, \dot {\tilde r}, w) = \frac{v(\dot {\tilde  t}, \dot {\tilde r}, w)}{u(\dot {\tilde  t}, \dot {\tilde r}, w)^2}\,,
			\end{align}
			for an arbitrary $\Xi=\Xi(z)$ independent of $\tilde t$ and $\tilde r$.
			\end{enumerate}
		\end{enumerate}

\end{theorem}

\bigskip

\textbf{Remark.} In the beginning of the theorem we assumed that $k_{10}\neq0$. If, instead, we have $k_7\neq0$ and $k_{10}=0$, a completely analogous theorem holds with the roles of $\dot t$ and $\dot r$ interchanged. The case when $k_{7}$ and $k_{10}$ are both zero implies that equation \eqref{eq:deltaw} is trivially satisfied and is discussed below.

\bigskip

\begin{theorem}[The case $\delta_w=0$]:\label{thm:main2}
	Let $(M,L = L(t,r,\dot t, \dot r, w))$ be an $SO(3)$-spatially symmetric pseudo-Finsler space and $\Gamma$, a spatially spherically symmetric affine connection on $M$, with connection coefficients \eqref{eq:ac1}- \eqref{eq:ac4} and curvature coefficients \eqref{a_i}. Assume $k_7=k_8=k_9=k_{10}=k_{11}=k_{12}=0$.
	
	Then: $(M,L)$ is of Berwald type and non-pseudo-Riemannian, with canonical connection given by $\Gamma$, if and only if one of the following conditions holds:
	\begin{enumerate}[1.]
		\item $[\delta_t, \delta_r] = 0$. In this case, up to a possible coordinate change $(t,r)\rightarrow (\tilde{t},\tilde{r})$, $L$ is an arbitrary 2-homogeneous function of $\dot{t},\dot{r},\dot{w}$ only:		\begin{align}\label{eq:2-L1}
			L = L(\dot t, \dot r, w) = \dot t^2 L(1, p, s)\,, \quad p = \frac{\dot r}{\dot t},\quad s = \frac{w}{\dot t}\,.
		\end{align}
	
		\item $[\delta_t, \delta_r]\neq0$,  $[\delta_t, [\delta_t, \delta_r]]\sim [\delta_t, \delta_r]$, $[\delta_r, [\delta_t, \delta_r]]\sim [\delta_t, \delta_r]$; in this case,
		\begin{align}\label{eq:2-L2}
				L =  w^2 \xi(q)\,,\quad q = \frac{\dot t e^{I-\varphi}}{w}\,,
		\end{align}
		with
		\begin{align}
			\varphi = \int_C (K dt + T dr), \quad I =\int \frac{(a_1+a_2 p)}{( a_2 p^2 - (a_4-a_1)p -a_3 )} dp\,
		\end{align}
 where, in the first case, integration is taken along an arbitrary curve $C$ connecting some initial point $(t_0,r_0)$ to $(t,r)$, and 
		\begin{align}
			K &= \partial_t I - (k_1 + k_2 p) + (k_1 p + k_2 p^2 - k_4 - k_6 p)\partial_p I\,,\\
			T &= \partial_r I - (k_2 + k_3 p) + (k_2 p + k_3 p^2 - k_6 - k_4 p)\partial_p I\,.
		\end{align}		
	\end{enumerate}

\end{theorem}

The proofs of these two theorems require to discuss numerous different cases, which makes them quite lengthy. Some remarks which should guide the reader through the proof:
\begin{itemize}
	\item Theorem \ref{thm:main}, point I:
	\begin{itemize}
		\item[a)] The necessity of this condition was proven above, in \eqref{eq:k11k12vanish}.
		\item[b)] will appear as necessary condition in Lemma \ref{lem:wlie}.
		\item[c)] will appear as necessary condition from Lemma \ref{lem:trLie}.
	\end{itemize}
	\item Theorem \ref{thm:main}, point II:
		The explicit expression for $L$, hence sufficient conditions, are derived in:
		\begin{itemize}
		\item[1.i.)]  Conclusion \ref{conc:pl}.
		\item[1.ii.)] Conclusion \ref{conc:exp}.
		\item[2.i.)]  Conclusion \ref{conc:XI1}.
		\item[2.ii.)] Conclusion \ref{conc:XI2}.
		\end{itemize}
	\item Theorem \ref{thm:main2}: 
	\begin{itemize}
		\item[1.] is derived in Conclusion  \ref{conc:notr}.
		\item[2.] is derived in Conclusion  \ref{conc:xi3}.
	\end{itemize}
\end{itemize}
Finally, we did only refer to non-Riemannian (non-quadratic) solutions, since it is well known that a quadratic Finsler function leads to geometry of Berwald type.

\section{Evaluating the Berwald condition: Proof of Theorems \ref{thm:main} and \ref{thm:main2}}\label{sec:mainpre}
Theorems \ref{thm:main} and \ref{thm:main2} will be proven by finding all possible $SO(3)$-symmetric solutions of the Berwald condition equations \eqref{eq:deltt}-\eqref{eq:homL}, for arbitrary $SO(3)$-spherically symmetric affine connections on $M$. According to Theorem \ref{thm:berw}, these are all possible Berwald-type pseudo-Finsler functions on our manifold, possessing the said symmetry.
The strategy to solve \eqref{eq:deltt} - \eqref{eq:homL} is as follows:
\begin{itemize}
	\item First, we note that there exist necessary consistency conditions, given by the vanishing of the action of the iterated Lie brackets $[\delta_i,\delta_j]L = 0, [\delta_k,[\delta_i,\delta_j]]L = 0,...$, $i,j,k \in \{t,r,w\}$. These Lie brackets are easily expressed in terms of the curvature and derivatives of the curvature of the connection.
	
	\item Second, we observe that: the equation $\delta_w L=0$, the homogeneity condition and the consistency conditions $[\delta_i,\delta_j]L = 0, [\delta_k,[\delta_i,\delta_j]]L = 0,...$ only involve the $\dot{t},\dot{r}$- and $w$-derivatives of $L$, hence we will use these equations to determine the dependence of $L$ on these three variables - accordingly, the $\dot{x^a}$-dependence of $L$. 
	
	\item Finally, we substitute the solutions of the above equations into the $\delta_t$ and $\delta_r$ equations \eqref{eq:deltt} and  \eqref{eq:deltr}, to determine the $t$ and $r$ dependence of $L$.
\end{itemize}
We list below the first five such iterated Lie bracket conditions (we will see below that these are actually sufficient to completely determine the $\dot{x}$-dependence of $L$, in any of the cases listed at the end of the previous section):
\begin{align}
 [\delta_t, \delta_r]L &= (a_1 \dot t + a_2 \dot r)\dot\partial_t L + (a_3 \dot t + a_4 \dot r)\dot\partial_r L + a_5 w \partial_w L = 0\,, \label{eq:curvtr}\\
 [\delta_w, \delta_t]L &= a_6 w \dot\partial_t L + a_7 w \dot\partial_r L + (a_8 \dot t + a_9 \dot r) \partial_w L = 0\,,\label{eq:curvtw}\\
 [\delta_w, \delta_r]L &= a_{10} w \dot\partial_t L + a_{11} w \dot\partial_r L + (a_{12} \dot t + a_{13} \dot r) \partial_w L = 0\,,\label{eq:curvrw}\\
 [\delta_t, [\delta_t, \delta_r]]L &= (A_1 \dot t + A_2 \dot r)\dot\partial_t L + (A_3 \dot t + A_4 \dot r)\dot\partial_r L + A_5 w \partial_w L = 0\,,\label{eq:curvttr}\\
 [\delta_r, [\delta_t, \delta_r]]L &= (B_1 \dot t + B_2 \dot r)\dot\partial_t L + (B_3 \dot t + B_4 \dot r)\dot\partial_r L + B_5 w \partial_w L = 0\label{eq:currtr}\,.
\end{align}
In the above, the coefficients $a_i$ are nothing but the curvature components of the nonlinear connection $N$, given by the decomposition $[\delta_{a},\delta_{b}]= R^{c}{}_{ab}\dot{\partial}_{c},  a,b,c \in \{t,r,\theta,\phi\}$; this can be easily seen as, using $k_{11}=k_{12}=0$ and \eqref{eq:rel_theta_w}, one gets: $[\delta_{t},\delta_{\theta}]= -\frac{\dot{\theta}}{w}\delta_{w}$. \\
The coefficients $A_i,B_i$ contain the first order partial derivatives of $a_i$; the explicit expressions of $a_i, A_i$ and $B_i$ (which are, ultimately, functions of $t$ and $r$), are displayed in terms of the connection coefficients and their derivatives in Appendix~\ref{app:curv}.

\bigskip

To summarize, we now have four original equations (\eqref{eq:deltt}- \eqref{eq:homL}) and five additional constraints (\eqref{eq:curvtr} to \eqref{eq:currtr}) which we want to solve. We will start by solving the "$\dot x$-system" consisting of the seven equations \eqref{eq:deltaw}, \eqref{eq:homL}, \eqref{eq:curvtr}-\eqref{eq:currtr}, since these equations do not involve any $t$ or $r$-derivatives. They can be treated as an \textit{algebraic} system for $\dot \partial_tL$, $\dot \partial_rL$ and $\partial_wL$. The latter implies the following remarks:

\begin{itemize}
	\item Solutions of these equations which require $\dot \partial_t L = 0$, $\dot \partial_r L = 0$ or $ \partial_w L = 0$ must be discarded, since they cannot lead to a non-degenerate Finslerian metric tensor. Hence, at most \textit{two} of the six $\dot x$-equations (\eqref{eq:deltaw}, \eqref{eq:curvtr} to \eqref{eq:currtr}) (that are homogeneous in these derivatives) can be linearly independent.
	
	In particular, if the $w$-equation $\eqref{eq:deltaw}$ is nontrivial, then, at most one of the curvature constraints \eqref{eq:curvtr} to \eqref{eq:currtr} can be independent of it.

	\item On the other hand, if the $w$-equation \eqref{eq:deltaw} is trivial, i.e. if $k_{7}=k_{8}=k_{9}=k_{10}=0$, then the curvature coefficients $a_5$ to $a_{13}$ vanish identically and only the equations \eqref{eq:curvtr}, \eqref{eq:curvttr} and \eqref{eq:currtr} can still be nontrivial.
\end{itemize}

Following the last remark, we distinguish two major cases: 
\begin{itemize}
	\item the $w$-equation $\eqref{eq:deltaw}$ is nontrivial, that is, $\delta_w \neq 0$ (whose solution is presented in Theorem \ref{thm:main});
	\item the $w$-equation $\eqref{eq:deltaw}$ is trivial, $\delta_w = 0$, presented in Theorem \ref{thm:main2}.
\end{itemize}
In the following, let us investigate separately these two cases.

\subsection{Case 1: nontrivial $w$-equation}\label{sec:mainCase1}
For this case we first solve \eqref{eq:deltaw} and \eqref{eq:homL}, before we work ourselves through the additional constraint equations.

\subsubsection{2-homogeneous solutions of the  $w$-equation \eqref{eq:deltaw}}
Since we exclude the case that all of the derivatives $\dot \partial_a L = 0$ can be zero, as discussed above, we assume that at least one of the connection coefficients $k_{7}$ or $k_{10}$ is nonzero. To fix things, we assume $k_{10}\neq0$. By the method of characteristics, we find that the general solution of  \eqref{eq:deltaw} is a free function of the variables $t,r,u,v$, where: 
\begin{align}
		u = \dot t - a \dot r, \quad v = c \dot r^2 - 2 b \dot t \dot r - w^2\,,
\end{align}
and
\begin{align}\label{eq:abc}
	a = \frac{k_7}{k_{10}},\quad b = \frac{k_8}{k_{10}},\quad c = \frac{k_9 k_{10} - k_7 k_8}{k_{10}^2}\,.
\end{align}
Moreover using the 2-homogeneity of $L$, \eqref{eq:homL}, we find
\begin{align}\label{eq:wsol1}
		L = u^2 \Xi(t,r,z),\quad z:=\frac{v}{u^2}\,.
\end{align}

\bigskip

In the particular case when $k_{7}=0$, this gives
\begin{align}\label{eq:wsol2}
L = \dot t^2 \Xi(t,r,z),\quad z:=\frac{v}{\dot t^2}\,.
\end{align}
This way, in the case when $k_{10}=0$, we obtain in a completely similar way (interchanging the roles of $t$ and $r$, $\dot t$ and $\dot r$ as well as $k_8$ and $k_9$).

Thus, we have solved the nontrivial necessary equation \eqref{eq:deltaw}  for $L$ (defining a Finsler function of Berwald type) and can now continue with the further constraints.

\subsubsection{Solving the Lie bracket equations involving $\delta_w$}
Next, we substitute the solutions of equations \eqref{eq:deltaw}- \eqref{eq:homL}, which we just derived, into the consistency conditions \eqref{eq:curvtw} and \eqref{eq:curvrw}, thus finding further necessary conditions for the existence of nontrivial solutions of the original system.

\begin{lemma}\label{lem:wlie}
	If  \eqref{eq:deltaw} is non trivial, then a nontrivially Finslerian (non-Riemannian) solution $L$ of the Berwald conditions \eqref{eq:deltt}-\eqref{eq:deltr} can only exist if
	\begin{align}\label{eq:twrw=w}
		[\delta_t, \delta_w] = \alpha \delta_w,\quad  [\delta_r, \delta_w] = \beta \delta_w
	\end{align}
	for some functions $\alpha=\alpha(t,r)$ and $\beta=\beta(t,r)$.
\end{lemma}\vspace{-5pt}

\begin{proof}[Proof of Lemma \ref{lem:wlie}]
	Starting from the solution  \eqref{eq:wsol1} which we found for the $\delta_w$ equation and using the variables $(t,r,\dot r, u, z)$, we can rewrite the $\dot x$-derivatives of $L$ as
	\begin{align}\label{eq:newvar}
	\begin{split}
		\dot\partial_t L &= 2 ( \dot r b - u z )\partial_z \Xi + 2 u \Xi\,,\\
		\dot\partial_r L &= 2 ( \dot r (a b + c) + u (a z +b) )\partial_z \Xi - 2 a u \Xi\,,\\
		w \partial_w L  &= -2 (\dot r^2 (2 ab + c) + 2 u \dot r b - u^2 z)\partial_z \Xi\,.
	\end{split}
	\end{align}
	Substituting the above relations into the $tw$-Lie bracket constraint \eqref{eq:curvtw}, we find:
	\begin{align}\label{eq:twXi}
		\dot{r}\left[ \left( a_{6}b+a_{7}\left( ab+c\right) -aa_{8}-a_{9}\right) \right] \partial_z\Xi+u\{  \left[ z\left( -a_{6}+aa_{7}\right) +\left(ba_{7}-a_{8}\right) \right] \partial_z\Xi+\left( a_{6}-aa_{7}\right) \Xi \} =0\,,
	\end{align}
	which decays into two separate equations, since, on the one hand, $\Xi$ is independent of $u$ and $\dot r$ and, on the other hand, we must discard the solution $\partial_z\Xi = 0$ (which would imply $\partial_wL=0$):
	\begin{align}
		 a_{6}b+a_{7}\left( ab+c\right) -aa_{8}-a_{9} &= 0 \,,\label{eq:DwNonTrivConst}\\
		  \left[ z\left( -a_{6}+aa_{7}\right) +\left(ba_{7}-a_{8}\right) \right] \partial_z\Xi+\left( a_{6}-aa_{7}\right) \Xi &= 0\,.\label{eq:XiDwNonTriv}
	\end{align}
	Assuming that $\left( a_{6}-aa_{7}\right)\neq0$, then the second equation has the general solution:
	\begin{align}
		\Xi = \varphi(t,r) (z ( a_{6}-a a_{7}) + \left(ba_{7}-a_{8}\right))\,,
	\end{align}
	where $\phi$ is a free function of $t$ and r.
	But, this implies that $L$ is quadratic in $\dot x$, i.e. pseudo-Riemannian. 
	Thus, nontrivially Finslerian solutions can only exist when \eqref{eq:XiDwNonTriv} is an identity, i.e.\ ,
	\begin{align}\label{eq:curvrel1}
		a_6 = a a_7 = \frac{k_7}{k_{10}}a_7,\quad a_8 = b a_7 = \frac{k_8}{k_{10}}a_7\,.
	\end{align}
	Using this in the constraint \eqref{eq:DwNonTrivConst} gives
	\begin{align}\label{eq:curvrel2}
		a_9 = (ab +c) a_7 = \frac{k_9}{k_{10}} a_7\,.
	\end{align}
	
	The same line of argument can applied to the $rw$-equation \eqref{eq:curvrw} and leads to the constraints
	\begin{align}\label{eq:curvrel3}
		a_{10} = \frac{k_7}{k_{10}}a_{11},\quad a_{12} = \frac{k_{8}}{k_{10}}a_{11},\quad a_{13} = \frac{k_9}{k_{10}}a_{11}\,.
	\end{align}
	
But, the latter relations tell us that:
\begin{align}
[\delta_w,\delta_t]=\frac{a_7}{k_{10}}\delta_w, \quad [\delta_w,\delta_t]=\frac{a_{11}}{k_{10}}\delta_w\,,
\end{align}
that is, the statement of the Lemma holds for $\alpha=\frac{a_7}{k_{10}}, \beta=\frac{a_{11}}{k_{10}}.$
 $[\delta_t,\delta_w]$ and $[\delta_r,\delta_w]$.
\end{proof}

\bigskip

The conditions stated by the above Lemma mean precisely that the Lie brackets involving $\delta_w$ must be proportional to $\delta_w$, hence they must not impose new constraints on $L$.
But, once these conditions are satisfied, it follows that further Lie brackets will automatically also be proportional to $\delta_w$:
\begin{align}
[\delta_r,[\delta_t, \delta_w]] = (\alpha\beta+\delta_r\alpha)\delta_w\,, \quad [\delta_t,[\delta_t, \delta_w]] = (\alpha^2+\delta_t\alpha)\delta_w\, etc.;
\end{align}
therefore, they will not impose any new constraints on $L$.

This Lemma proves the necessity of the first part of condition $I.b)$ in Theorem \ref{thm:main}.\vspace{10pt} To prove its second part, we will have to integrate equations \eqref{eq:curvtr}, \eqref{eq:curvttr}.

\subsubsection{Solving the Lie bracket involving $\delta_t$ and $\delta_r$}
In order to find nontrivial real Finslerian solutions, we assume that $k_{11}=k_{12}=0$ and the necessary conditions given by Lemma \ref{lem:wlie} hold. A further necessary condition for the existence of solutions of our system \eqref{eq:deltt}-\eqref{eq:homL} is $[\delta_t, \delta_r]L=0$, that is \eqref{eq:curvtr}, which we will solve in the following.

\bigskip

Thus, we consider equation~\eqref{eq:curvtr} in the variables $(t,r,\dot r,u, z)$ and use \eqref{eq:newvar} to express it as
\begin{align}\label{eq:trxi}
	A \dot{r}^{2} \partial_z \Xi + \dot{r }u \big( \left( B z + C \right) \partial_z\Xi - B \Xi \big) + u^{2} \big( \left( Dz+ E \right) \partial_z \Xi - F \Xi \big) =0\,,  
\end{align}
where the coefficients $A,B,...,F$ are functions of $t$ and $r$ only. More precisely, 
\begin{align}\label{A-F_def}
\begin{split}
A &:=b\left( aa_{1}+a_{2}\right) +\left( ab+c\right) \left(aa_{3}+a_{4}\right) -a_{5}\left( 2ab+c\right)\,,   \\
B &:=a\left( aa_{3}+a_{4}\right) -\left( aa_{1}+a_{2}\right)\,, \\
C &:=\left( ab+c\right) a_{3}+b\left( aa_{3}+a_{4}\right) +b (a_{1}-2a_{5})\,, \\
D &:= aa_{3}-a_{1}+a_{5}\,, \\
E &:=ba_{3}\,, \\ 
F &:=aa_{3}-a_{1}\,.
\end{split}
\end{align}

A first question to answer is when does equation \eqref{eq:curvtr}, re-expressed as \eqref{eq:trxi}, impose an independent restriction compared to \eqref{eq:deltaw}. The answer is given by the following Lemma.

\begin{lemma} \label{lem:AF_0}
Assume $k_{10}\neq 0$. Then, the following statements are equivalent:
\begin{enumerate}
	\item $A=B=C=D=E=F=0$
	
	\item $\left[ \delta _{t},\delta _{r}\right] = \alpha \delta _{w},$ where either $\alpha = 0$ or $b=c=0$.
\end{enumerate}
\end{lemma}

\begin{proof}
	$1\Rightarrow 2:$ From $F=0,D=0$ we find: 
	\begin{equation}
	a_{1}=aa_{3}, ~\ \ a_{5}=0.  \label{eq:a1_a3}
	\end{equation}%
	Further, using (\ref{eq:a1_a3}) into $B=0$ gives:%
	\[
	a_{2}=aa_{4}.
	\]
	Equation $E=0$ then leads to two possibilities:
	\begin{enumerate}
			\item  If $a_{3}\not=0,$ then, necessarily $b=0.$ Then, $C=0$ gives: $c=0$; collecting the results, we have:%
		\[
		\left[ \delta _{t},\delta _{r}\right] =(a_{3}\dot{t}+a_{4}\dot{r})(a\dot{%
			\partial}_{t}+\dot{\partial}_{r}),~\ \ \ \delta _{w}=k_{10}w(a\dot{\partial}%
		_{t}+\dot{\partial}_{r}),
		\]%
		therefore, $\left[ \delta _{t},\delta _{r}\right] $ and $\delta _{w}$ are
		proportional, with proportionality factor: $\alpha =\dfrac{(a_{3}\dot{t}%
			+a_{4}\dot{r})}{k_{10}w}.$
		
		\item  $a_{3}=0,$ which, using the remaining equations, gives $a_{1}=a_{2}=a_{3}=a_{4}=a_{5}=0,$ (that is, $\left[ \delta _{t},\delta _{r}\right] =0\delta _{w}$), or, again $b=c=0$ (which is similar the former case).
		\end{enumerate}

	$2\Rightarrow 1:$ The proportionality hypothesis $\left[ \delta _{t},\delta
	_{r}\right] =\alpha \delta _{w}$ (where $\alpha $ can depend both on $x$ and 
	$\dot{x}$) means:%
	\[
	a_{1}\dot{t}+a_{2}\dot{r}=\alpha k_{7}w;~\ \ a_{3}\dot{t}+a_{4}\dot{r}%
	=\alpha k_{10}w,~\ \ \ a_{5}w=\alpha \left( k_{8}\dot{t}+k_{9}\dot{r}\right)
	.
	\]%
	Eliminating $\alpha $ between the first two relations and taking into
	account that $a_{i},k_{i}$ do not depend on $\dot{t},\dot{r}$ or $w,$ yields:%
	\begin{equation}
	a_{1}=aa_{3},~\ \ a_{2}=aa_{4},  \label{eq:a1a3a2a4}
	\end{equation}%
	(where we recall that $a=\tfrac{k_{7}}{k_{10}}$). Then, elimination of $%
	\alpha $ between the latter two relations leads to the equations:%
	\[
	k_{8}a_{3}=0,~\ \ k_{9}a_{4}=0,~\ \ \ k_{9}a_{3}+k_{8}a_{4}=0,~\ \ a_{5}=0.
	\]%
	We obtain two possibilities:
	\begin{enumerate}
		\item $a_{3}=a_{4}=0,$ which then leads to $a_{1}=...=a_{5}=0$, which then immediately implies statement 1.
		\item $k_{8}=k_{9}=0,$ which means $b=c=0.$ Together with (\ref{eq:a1a3a2a4})
		and $a_{5}=0,$ these give, again, $A=B=C=D=E=F=0$.
	\end{enumerate}
This concludes the proof.
\end{proof}

\bigskip

Now we are ready to integrate equation \eqref{eq:trxi}; its integration leads to the following lemma.
\begin{lemma}\label{lem:trLie}
	Let $k_{10}\neq 0$ and $A,B,C,D,E,F$ be as in \eqref{A-F_def}. Then:
	\begin{enumerate}
		\item A necessary condition for the Berwald conditions \eqref{eq:deltt}-\eqref{eq:homL} to admit a non-Riemannian Finslerian solution is that
		\begin{align}\label{eq:connconst}
		A=B=C=0\,.
		\end{align}
		\item Moreover, if $D\neq0$, then the solution (if any), is of the form $L = u^2 \Xi(t,r,z)$, with
		\begin{align}\label{eq:Xi1}
		\Xi(t,r,z) = \varphi(t,r) (D z + E)^{\frac{F}{D}}\,,
		\end{align}
		whereas if $D=0,E\neq 0$, this can only be of the form
		\begin{align}\label{eq:Xi2}
		\Xi(t,r,z) = \varphi(t,r) e^{z\frac{F}{E}}\,,
		\end{align}
		where $\varphi$ is an arbitrary function of $t$ and $r$.
		\item If $D=E=0$, then, solutions can only exist when $F=0$ (that is, $[\delta_t, \delta_r] \sim \delta_w$) and:
		\begin{itemize}
			\item  $b,c$ are not both zero. In this case, upon a coordinate change $(t,r)\rightarrow (\tilde{t},\tilde{r})$, $L$ must be independent of the new $t$ and $r$-coordinates, $\tilde t$ and $\tilde r$:
			\begin{align}
				L = L(\dot{\tilde t}, \dot{\tilde r}, w) = u^2 \Xi(z)\,;
			\end{align}
			\item $b=c=0$ and $a\neq0$. In this case, $L$ must be of the form 
			\begin{align}\label{eq:Xi3}
				L = u^2 \Xi(t,r,z)\,,
			\end{align} 
			for an arbitrary function $\Xi = \Xi(t,r,z)$ and $z = -w^2/u^2$.
		\end{itemize}
	\end{enumerate} 	
\end{lemma}

\begin{proof}[Proof of Lemma \ref{lem:trLie}]
	\begin{enumerate}
		\item As $\Xi$ does not depend on $\dot{r}$ or $u$, in equation \eqref{eq:trxi}, each term multiplying different powers of $\dot r$ and $u$ must vanish separately; discarding, again the option $\partial_z\Xi=0$, we find:
		\begin{align}
		A &= 0\,,  \label{tr-eq-split-1} \\
		\left( Bz+C\right) \partial_z\Xi-B\Xi  &=0\,,  \label{tr-eq-split-2} \\
		\left( Dz+E\right) \partial_z\Xi-F\Xi  &=0\,.  \label{tr-eq-split-3}
		\end{align}
		Hence, analogously to Equation \eqref{eq:XiDwNonTriv}, if $B\neq0$ the solution for $\xi$ is of the form
		\begin{align}
		\Xi = \varphi(t,r) (z B + C)\,,
		\end{align}
		which implies a quadratic pseudo-Riemannian form of $L$. Thus, for a non-Riemannian solution, we must set $B=0$, which implies immediately also $C=0$. Thus, we have proven the first statement of the lemma.
		\item Assuming that $D\neq 0$, integration of \eqref{tr-eq-split-3} gives \eqref{eq:Xi1}. For the case $D=0$ and $E\neq0$, one immediately obtains \eqref{eq:Xi2}, which proves the second statement of the lemma. 
		\item Assume $D=E=0$. The case $E=D=0$ and $F\neq0$ leads to $\Xi=0$, which is not a valid solution. Therefore, we must necessarily have $F=0$. Thus, we actually have $A=B=C=D=E=F=0$, which, according to Lemma \ref{lem:AF_0} means that $\left[ \delta _{t},\delta _{r}\right] = \alpha \delta _{w}$, for some $\alpha=\alpha(t,r,\dot{t},\dot{r},w)$; more precisely:
		\begin{itemize}
		\item If $b,c$ are not both zero (which is the same as: $k_8,k_9$ are not both zero), then $[\delta_t,\delta_r]=0$. But, by Frobenius' theorem this means that the integral curves of $\delta_t$ and $\delta_r$ can be used to define new local coordinates $\tilde t$ and $\tilde r$, and thus $\delta_t L= 0$ and $\delta_r L = 0$ imply that $L$ is independent of $\tilde t$ and $\tilde r$.
		\item For the case when $b=c=0$, the proportionality $\left[ \delta _{t},\delta _{r}\right] = \alpha \delta _{w}$ tells us that equations \eqref{eq:deltr} and \eqref{eq:deltaw} are equivalent. Since we already ensured that $L = u^2 \Xi(t,r,z)$ solves the equation $\delta_w L = 0$, see \eqref{eq:wsol1}, this is the general solution and $\Xi(z)$ remains free. Moreover, $b=c=0$ implies $k_8=k_9=0$, thus we find the variable $v=-w^2$ and hence $z=-w^2/u^2$. This proves the third statement of the lemma.
		\end{itemize}
	\end{enumerate}
\end{proof}\vspace{10pt}



The above Lemma shows, in the case of a nontrivial $w$-equation, necessary consistency conditions for the Berwald condition equations, given by the Lie bracket $[\delta_t,\delta_r]$. Let us investigate, in the following, if (and when) do further Lie brackets involving $\delta_t$ and $\delta_r$ impose new consistency conditions.

\subsubsection{The iterated Lie brackets involving $\delta_t$ and $\delta_r$}
Assume, in the following that $k_{11}=k_{12}=0$ and the necessary conditions given by Lemmas \ref{lem:wlie} and \ref{lem:trLie} are satisfied. \\

The iterated Lie bracket equations $ [\delta_t, [\delta_t, \delta_r]]L=0$, \eqref{eq:curvttr}, and $ [\delta_t, [\delta_t, \delta_r]]L=0$, \eqref{eq:currtr}, take a similar form as \eqref{eq:trxi} in the previous section, just with coefficients constructed from the functions $A_1$ to $A_5$ or $B_1$ to $B_5$, defined in appendix \ref{app:curv}, respectively. Hence their treatment is analogous to the system \eqref{tr-eq-split-1}- \eqref{tr-eq-split-3} and we find the necessary conditions
\begin{align}
	\mathcal{A} = \mathcal{B} = \mathcal{C} &= 0\,, \\
	(\mathcal{D}z+\mathcal{E})\partial_z\Xi - \mathcal{F}\Xi &= 0 \label{eq:ttr-eq-split-3}\,,
\end{align}
where, e.g., for $ [\delta_t, [\delta_t, \delta_r]]L$, we have
\begin{align}\label{eq:defmA-mF}
\begin{split}
	\mathcal{A} &:= b\left( aA_{1}+A_{2}\right) +\left( ab+c\right) \left(aA_{3}+A_{4}\right) -\left( 2ab+c\right) A_{5}\\
	\mathcal{B} &:= a\left( aA_{3}+A_{4}\right) -\left( aA_{1}+A_{2}\right) \\
	\mathcal{C} &:= \left( ab+c\right) A_{3}+b\left( aA_{3}+A_{4}\right) + bA_{1}-2bA_{5}\\
	\mathcal{D} &:= aA_{3}-A_{1}+A_{5}\\
	\mathcal{E} &:= bA_{3}\\
	\mathcal{F} &:= aA_{3}-A_{1}\,.
\end{split}
\end{align}

Comparing \eqref{tr-eq-split-3} and \eqref{eq:ttr-eq-split-3} we find that they can only be compatible in two cases: either equation \eqref{tr-eq-split-3} is trivial, that is
\begin{align}
	D = E = F = 0
\end{align}
or:
\begin{align} \label{eq:mathcal_DEF}
	\mathcal{D} = \alpha D,\quad \mathcal{E} = \alpha E,\quad \mathcal{F} = \alpha F\,,
\end{align}
for some function $\alpha=\alpha(t,r,\dot t, \dot r, w)$.
\bigskip
Let us clarify the meaning of these two cases.
\begin{itemize}
\item The first case, $D = E = F = 0$, corresponds to statement 3 of Lemma \ref{lem:trLie}, in particular, $[\delta_t,\delta_r]\sim \delta_w$. But, the latter always entails, taking into account  \eqref{eq:twrw=w}, that  $[\delta_t,[\delta_t,\delta_r]]\sim \delta_w$ and $[\delta_t,[\delta_t,\delta_r]]\sim \delta_w$ , hence the double Lie brackets will automatically not impose any new constraints. 
\item The second case (with $D,E,F$ not all zero - that is, $[\delta_t,\delta_r]$ and $\delta_w$ are linearly independent), corresponds to statement 2 of Lemma \ref{lem:trLie}; in this case, the $\dot{x}$-dependence of $L$ is already completely specified. Hence, in order to have a valid solution of the system \eqref{eq:deltt} to \eqref{eq:homL} (which, we recall, requires that none of the $\dot{x}$-derivatives of $L$ can be zero), the double Lie brackets must necessarily be linear combinations of $[\delta_t,\delta_r]$ and $\delta_w$. Actually, we show in the Lemma below that \eqref{eq:mathcal_DEF} gives a much sharper statement.
\end{itemize}

\begin{lemma}\label{lem:ttrLie}
Assume $k_{10}\neq 0$ and $A,B,C,D,E,F$ are such that $A=B=C=0$, but $D,E,F$ are not all zero. Then, the consistency conditions
\begin{align}\label{eq:ABCDEFhyp}
\quad \mathcal{A}=\mathcal{B}=\mathcal{C} = 0,\quad 
\mathcal{D} = \alpha D,\quad \mathcal{E} = \alpha E,\quad \mathcal{F} = \alpha F\,,
\end{align}
for some $\alpha=\alpha(t,r,\dot t, \dot r, w)$, are equivalent to: 
\begin{align}\label{eq:ttr-tr}
A_i = \alpha a_i, \quad i=1,2,3,4,5\,,
\end{align}
which means:
\begin{align}
[\delta_t,[\delta_t,\delta_r]] \sim [\delta_t,\delta_r].
\end{align}
\end{lemma}

\begin{proof}[Proof of Lemma \ref{lem:ttrLie}]
	
$\rightarrow$: Assume that \eqref{eq:ABCDEFhyp} are satisfied and let us start by two remarks. First, comparing the relations $	\mathcal{D} = \alpha D,\quad \mathcal{F} = \alpha F$, we find:
\begin{align}
A_5=\alpha a_5.
\end{align}
Second, if $a_5=0$, that is, $D=F$, then, taking into account \eqref{tr-eq-split-3}, the solution, if any, is Riemannian - hence we will discard this case and assume that $a_5 \neq 0$. \\

But, imposing that $a_5\neq 0$, it follows from its expression \eqref{a_i} that $k_8$ and $k_9$ cannot vanish simultaneously; therefore, $b$ and $c$ cannot be simultaneously zero.\vspace{10pt}

If $b\neq 0$, we find from $\mathcal{E} = \alpha E$ that $A_3=\alpha a_3$ and then, from $\mathcal{F} = \alpha F$ that also $A_1=\alpha a_1$. The other two relations \eqref{eq:ttr-tr} follow then immediately from the remaining hypotheses \eqref{eq:ABCDEFhyp}. \vspace{10pt}

If $b=0$, then, taking into account that $c\neq0$, equalities $\mathcal{C}=0, C=0$ give respectively $A_3=0, a_3=0.$ Using these equalities and $A_5=\alpha a_5$ in the remaining hypotheses \eqref{eq:ABCDEFhyp} gives, again, \eqref{eq:ttr-tr}.

$\leftarrow$: Assuming $A_i = \alpha a_i$, relations \eqref{eq:ABCDEFhyp} follow immediately.
\end{proof}\vspace{10pt}

Similarly, we get the consistency condition:
\begin{align}
[\delta_r,[\delta_t,\delta_r]]\sim [\delta_t,\delta_r].
\end{align}

\textbf{Remark.} The proportionalities $[\delta_t,[\delta_t,\delta_r]]\sim [\delta_t,\delta_r]$, $[\delta_r,[\delta_t,\delta_r]]\sim [\delta_t,\delta_r]$ do not, in general, hold automatically (they only seem to hold automatically if we are lucky enough as to have $[\delta_t,\delta_r]\sim \delta_w$); hence one must impose them. Yet, once imposed, these two proportionality relations ensure that further iterated Lie brackets do not yield any new constraints.

\bigskip

This way, we have proven the necessity of the latter part of condition I.b) in Theorem \ref{thm:main} . The next step is to take the solutions we found as ansatzes and explicitly find the solutions of the remaining  equations $\delta_t L = 0$, \eqref{eq:deltt}, and $\delta_r L = 0$, \eqref{eq:deltr}. This will provide necessary \textit{and sufficient} conditions. \vspace{10pt}

Before we do so, let us summarize what we have found so far. For a spherically symmetric, non-Riemannian Berwald Finsler manifold, the Finsler function (if any) must be of one of the three kinds:
\begin{enumerate}
	\item $L =  u^2 \Xi(t,r,z),\ z = v/u^2$, $D\neq0$, $[\delta_t,\delta_r]$ not proportional to $\delta_w$
	\begin{align}
			\Xi(t,r,z) = \varphi(t,r) (D z + E)^{\frac{F}{D}}.
	\end{align}
	\item $L =  u^2 \Xi(t,r,z),\ z = v/u^2$, $D=0, E\neq0$, $[\delta_t,\delta_r]$ not proportional to $\delta_w$
	\begin{align}
			\Xi(t,r,z) = \varphi(t,r) e^{z\frac{F}{E}}.
	\end{align} 
	\item $L =  u^2 \Xi(t,r,z),\ z = v/u^2$, $D=E=F=0$, then $[\delta_t, \delta_r] = \alpha\delta_w$, and $\Xi(t,r,z)$ is arbitrary.
\end{enumerate}

\subsubsection{The $t$- and $r$-equations}
In this section, we will solve the remaining equations $\delta_t L= 0 = \delta_r L$, for $L$ of the form $L = u^2 \Xi(t,r,z)$; an important simplification will be achieved by employing the necessary relations \eqref{eq:curvrel1}-\eqref{eq:curvrel3} between different curvature coefficients given by Lemmas \ref{lem:wlie}, \ref{lem:trLie} and \ref{lem:ttrLie}.

To evaluate these equations for $L = u^2 \Xi(z)$ we use \eqref{eq:newvar} and 
\begin{align}
	\partial_t L &= -2\dot{r}u \Xi \partial_t a + u^{2} \partial_{t} \Xi + \partial_z\Xi \left[ \dot{r}^{2} \left( \partial_t c + 2a \partial_t b \right) + 2 \dot{r} u \left( \partial_t b + z \partial_t a \right) \right]\,,  \label{t-eq-z} \\
	\partial_r L &= -2 \dot{r} u \Xi \partial_ra + u^{2} \partial_{r}\Xi + \partial_z \Xi \left[ \dot{r}^{2} \left( \partial_r c + 2a \partial_r b \right) + 2 \dot{r} u \left( \partial_r b + z \partial_r a \right) \right]\,.  \label{r-eq-z}
\end{align}
Combining all the terms, the equations can be grouped as follows
\begin{align}
	\delta_t L= 0 \Leftrightarrow P_1 \dot r^2 + 2 Q_1 \dot r u + R_1 u^2 = 0\\
	\delta_r L= 0 \Leftrightarrow P_2 \dot r^2 + 2 Q_2 \dot r u + R_2 u^2 = 0\,.
\end{align}
where $P_i, Q_i, R_i, i=1,2$ are functions of $t,r$ and $z$ only. Thus, each of these coefficients must vanish separately. 

Reading of these coefficients from the equations yields:
\begin{align}
	P_{1} &= \partial_z \Xi\left[ \partial_t c+2a\partial_t b-2b\left( k_{2}+k_{1}a\right) - 2\left( k_{6}+k_{4}a\right) \left( c+ab\right) +2k_{8}\left( 2ab+c\right) \right]\,, \\
	P_{2} &= \partial_z \Xi\left[ \partial_r c+2a\partial_r b-2b\left( k_{3}+k_{2}a\right) -2\left( k_{5}+k_{6}a\right) \left( c+ab\right) +2k_{9}\left( 2ab+c\right) \right]\,, \\
	Q_{1} &=\partial_z \Xi\left\{ z\left[ \partial_t a-a\left( k_{6}+k_{4}a\right) + \left( k_{2}+k_{1}a\right) \right] +\left[ \partial_t b-k_{1}b-\left( k_{6}+k_{4}a\right) b-k_{4}\left( c+ab\right) +2k_{8}b\right] \right\} \\
	&+\Xi \left[ -\partial_t a+a\left( k_{6}+k_{4}a\right) -\left(k_{2}+k_{1}a\right) \right]\,, \\
	Q_{2} &=\partial_z \Xi\left\{ z\left[ \partial_r a-a\left( k_{5}+k_{6}a\right) + \left( k_{3}+k_{2}a\right) \right] +\left[ \partial_r b-k_{2}b-\left(k_{5}+k_{6}a\right) b-k_{6}\left( c+ab\right) +2k_{9}b\right] \right\} \\
	&+\Xi \left[ -\partial_r a+a\left( k_{5}+k_{6}a\right) -\left(
	k_{3}+k_{2}a\right) \right] , \\
	R_{1} &=\partial _{t}\Xi +\partial_z \Xi\left[ z\left(2k_{1}-2k_{4}a-2k_{8}\right) -2k_{4}b\right] +\Xi \left(2k_{4}a-2k_{1}\right) , \label{R1} \\
	R_{2} &=\partial _{r}\Xi +\partial_z \Xi \left[ z\left(2k_{2}-2k_{6}a-2k_{9}\right) -2k_{6}b \right] +\Xi \left(2k_{6}a-2k_{2}\right)\,.\label{R2}
\end{align}
Employing the definition of $a,b$ and $c$, see \eqref{eq:abc}, $P_1, P_2, Q_1$ and $Q_2$ can be expressed in terms of the curvature coefficients \eqref{a_i} as:
\begin{align}
	k_{10}^{2}P_{1} &= \partial_z\Xi\left[-a_{8}k_{7}+a_{6}k_{8}-a_{9}k_{10}+a_{7}k_{9}\right]\,,\label{P1}\\
	k_{10}^{2}P_{2} &=\partial_z\Xi \left[-a_{12}k_{7}+a_{10}k_{8}-a_{13}k_{10}+a_{11}k_{9}\right]\,,\label{P2}\,, \\
	k_{10}^{2}Q_{1} &=\partial_z\Xi \left[ z\left(a_{7}k_{7}-a_{6}k_{10}\right) +\left( a_{8}k_{10}-a_{7}k_{8}\right) \right] -\Xi \left( a_{7}k_{7}-a_{6}k_{10}\right)\,,  \label{Q1} \\
	k_{10}^{2}Q_{2} &=\partial_z\Xi \left[ z\left(a_{11}k_{7}-a_{10}k_{10}\right) +\left( a_{12}k_{10}-a_{11}k_{8}\right)\right] 
	-\Xi \left( a_{11}k_{7}-a_{10}k_{10}\right)\,.  \label{Q2}
\end{align}

\begin{lemma}\label{lem:teq-req}
	If the consistency conditions \eqref{eq:curvrel1}-\eqref{eq:curvrel3} hold, then the equations $\delta_t L= 0 = \delta_r L$ are equivalent to 
	\begin{align}
		\partial_t (\ln \Xi) &= M - \partial_z (\ln \Xi) \left(\tilde M z - 2 k_4 b\right) \,, \label{eq:t-M}\\
		\partial_r (\ln \Xi) &= N - \partial_z (\ln \Xi) \left(\tilde N z - 2 k_6 b\right)\,, \label{eq:r-N}
	\end{align}
	with 
	\begin{align}
		M &:= 2 (k_1 - k_4 a)\,, \quad \tilde M := M - 2k_8\,,\\
		N &:= 2 (k_2 - k_6 a)\,, \quad \tilde N := N - 2k_9\,. 
	\end{align}
\end{lemma}

\begin{proof}[Proof of Lemma \ref{lem:teq-req}]
	The proof is straightforward. Using \eqref{eq:curvrel1}-\eqref{eq:curvrel3} in \eqref{P1}, \eqref{P2}, \eqref{Q1} and \eqref{Q2} one finds that the relations $P_1=P_2=Q_1=Q_2 = 0$ are identically satisfied. The remaining equations then are $R_1=R_2=0$, which are precisely \eqref{eq:t-M} and \eqref{eq:r-N} as can be seen by comparison with \eqref{R1} and \eqref{R2}.
\end{proof}

\bigskip

The abbreviations $M$ and $N$ are convenient since
\begin{align}\label{eq:DEFMN}
	F = \frac{1}{2}(\partial_t N - \partial_r M)\,,\quad D = \frac{1}{2}(\partial_t \tilde N - \partial_r \tilde M)\,,\quad E = b(k_6 \tilde M - k_4 \tilde N) + \partial_r (k_4 b) - \partial_t (k_6 b)\,.
\end{align}

Having done this preparatory work, let us consider the different possible classes of solutions identified by Lemma \ref{lem:trLie}.

\begin{enumerate}
	\item \textbf{Power law solutions}: $L =  u^2 \Xi(z),\ z = v/u^2$, with
	\begin{align}\label{eq:pl}
		\Xi(t,r,z) = \varphi(t,r) (E + D z)^{\frac{F}{D}} =: \vartheta(t,r)(z + \rho)^\lambda\,.
	\end{align}
	These were obtained as solutions of the curvature constraints, in the case 
	$D\neq0$, $[\delta_t,\delta_r]$ not proportional to $\delta_w$. Also, we have
	set $\lambda = \frac{F}{D}$, factored $D^\lambda$ and absorbed the appearing powers into $\vartheta = \varphi D^\lambda$ and $\rho = \frac{E}{D}$. 
	
	\noindent In the following, we will use the above as an ansatz for the remaining Berwald conditions, i.e., for \eqref{eq:t-M} and \eqref{eq:r-N}. 
	Employing \eqref{eq:pl} in \eqref{eq:t-M} and \eqref{eq:r-N}, multiplying by a factor $(z+\rho)$ and applying two derivatives w.r.t. $z$ on the equations implies that
	\begin{align}
	\partial_t \lambda =\partial_r \lambda  = 0\,.
\end{align}
Thus, the fraction $\lambda  = \frac{F}{D}$ must be a constant, independent of $t$ and $r$. 
But, these conditions are automatically satisfied. To see this, we first note that equation $\partial_t \lambda = 0$ is equivalent to:
\begin{equation}
k_{7}^{2}\left( A_{1}a_{5}-A_{5}a_{1}+\frac{1}{k_{10}}a_{3}a_{5}\left(
a_{6}-aa_{7}\right) -Ba_{5}k_{4}+a(a_{3}A_{5}-A_{3}a_{5})\right) =0\,;
\end{equation}
then, Lemma \ref{lem:wlie} ensures that $a_{6}-aa_{7}=0$, Lemma \ref{lem:trLie} tells us that $B=0$, whereas, by Lemma \ref{lem:ttrLie}, we have $ A_{1}a_{5}-A_{5}a_{1}=0, a_{3}A_{5}-A_{3}a_{5}=0$.
The identity $\partial_r \lambda=0$ follows in a completely similar manner.

\bigskip

Using $\lambda = const.$, and observing that all the involved functions are independent of $z$, it follows that the coefficients of the different powers of $z$ in \eqref{eq:t-M} and \eqref{eq:r-N} must vanish separately. This yields:
\begin{align}\label{eq:vartheta}
	\partial _{t}(\ln \vartheta ) = M-\lambda \tilde{M}\,,\quad
	\partial _{r}(\ln \vartheta ) = N-\lambda \tilde{N}\,,
\end{align}
and: 	
 \begin{align}
\rho \partial _{t}(\ln \vartheta )+\lambda \rho _{,t}  = M\rho +2\lambda k_{4}b\,,\quad
\rho \partial _{r}(\ln \vartheta )+\lambda \rho _{,r}  = N\rho +2\lambda k_{6}b\,,
 \end{align}	
The fact that $\lambda = F/D$, together with \eqref{eq:DEFMN}, ensure that $\partial_r(M-\lambda\tilde{M})=\partial_t(N-\lambda\tilde{N})$, that is, the consistency of equations \eqref{eq:vartheta} which thus determine $\vartheta$ as:
\begin{align}\label{vartheta_power}
		\vartheta(t,r) = \exp\left(\int (M - \lambda \tilde M)dt\right) =  \exp\left(\int (N - \lambda \tilde N)dr\right)\,.
\end{align}
Substituting the found solution into the second one, we obtain for $\rho$ the equations
\begin{align}\label{rho_power}
		\partial_t \rho = 2 k_4 b + \tilde M \rho\,,\quad \partial_r \rho = 2 k_6 b  + \tilde N \rho\,.
\end{align}
But, $\rho$ is already known, as $\rho = E/D$. Substituting this value into the first equation above, using the definitions of $B,C,D,E,F,\tilde{M}$ and the expressions of the derivatives $a_{i,t}$ in terms of the double Lie bracket coefficients $A_i$, we find that this is actually, equivalent to:
\begin{align}
	\left( a_{8}k_{10}-a_{7}k_{8}\right)a_{3}(a_{1}+a_{5})
	+a_{3}^{2}\left( a_{6}k_{8}-a_{8}k_{7}\right) \\
	+ k_{4}k_{10}^{2}(CD-BE)+
	bk_{10}\left( A_{1}a_{3}-A_{3}a_{1}\right) +bk_{10}\left(
	A_{3}a_{5}-A_{5}a_{3}\right) =0\,;
\end{align}
the first two terms vanish by Lemma \ref{lem:wlie}, the third one, using $B=C=0$ (by Lemma \ref{lem:trLie}) and the last two ones, by Lemma \ref{lem:ttrLie}.

\begin{conclusion}[Power law]\label{conc:pl}
	Assuming all the conditions in Lemmas \ref{lem:wlie},  \ref{lem:trLie} and \ref{lem:ttrLie} are satisfied, then $\lambda:= F/D$ is a constant and the pseudo-Finsler function $L = u^2 \Xi$ defined by \eqref{eq:pl} (with  $\vartheta$ defined by the integral \eqref{vartheta_power} and $\rho=E/D$) is of Berwald type, with canonical connection $\Gamma$. Thus, we proved statement II.1.i.) of Theorem~\ref{thm:main}.
\end{conclusion} 

\bigskip

	\item \textbf{Exponential solutions}: $L =  u^2 \Xi(z),\ z = v/u^2$, where:	 
	\begin{align}\label{eq:exp}
		\Xi(t,r,z) = \varphi(t,r) e^{z\frac{F}{E}}\,.
	\end{align} 
	These were obtained for: $D=0, E\neq0$ (thus $b\neq0$ and $a_3\neq0$), $[\delta_t,\delta_r]$ not proportional to $\delta_w$.
	
	 Evaluating the system $A=B=C=D=0$ from \eqref{A-F_def} and solving it for $a_1,a_2,a_4$, one immediately finds:
	 \begin{align}
	 	2 ab +c =0\,,\quad a_1 = a a_3 + a_5\,,\quad a_2 = -a^2 a_3\,,\quad a_4 = -a a_3 + a_5\,. 
	 \end{align}
 	
	We are now ready to plug the ansatz \eqref{eq:exp} into the remaining Berwald conditions \eqref{eq:t-M} and \eqref{eq:r-N}.
	Doing this and realizing that the coefficients in front of the different powers of $z$ need to vanish individually, the resulting equations are (setting $F/E=:\mu$)
 	\begin{align}
 		\partial_t \ln \mu &= -\tilde M,& \partial_r \ln \mu &= -\tilde N\,, \label{eq:mu}\\
 		\partial_t \ln \varphi &= M + 2 k_4 b \mu,& \partial_r \ln \varphi  &= N + 2 k_6 b \mu\,.\label{eq:varphi} 
 	\end{align}
 	This system turns out to be always consistent due to $\mu = F/E$, $D=0$ and the relations \eqref{eq:DEFMN}, as follows. First, we note that $\mu = F/E$ satisfies the two relations \eqref{eq:mu} identically; then, using \eqref{eq:mu} and the last equality \eqref{eq:DEFMN}, we find that \eqref{eq:varphi} consistently determines $\varphi$ as a function of $t$ and $r$ from the connection coefficients, more precisely:
 	\begin{align}\label{eq:phi}
 		\varphi = e^{\int (M + 2 k_4 b \mu)dt} = e^{\int (N + 2 k_6 b \mu) dr}, \quad \mu:=F/E \,.
 	\end{align}
 	
  	\bigskip
  	
	\begin{conclusion}[Exponential solutions]\label{conc:exp}
	Assuming all the conditions in Lemmas \ref{lem:wlie}, \ref{lem:trLie} and \ref{lem:ttrLie} are satisfied, then, pseudo-Finsler functions $L = u^2 \Xi$ given by \eqref{eq:exp}, with $\mu=F/E$ and $\varphi$ defined by \eqref{eq:varphi}, are of Berwald type, with canonical connection $\Gamma$. Thus, we proved statement II.1.ii.) of Theorem \ref{thm:main}.
 	\end{conclusion} 
 
   	\bigskip
   	
 	\item[3a.] \textbf{Arbitrary $\Xi$  ($\boldsymbol{[\delta_t,\delta_r] = \alpha \delta_w$}):}
   	 $L =  u^2 \Xi(t,r,z),\ z = v/u^2$, $D=E=F=0$, $b=0$ and $c=0$. 
 	 Thus: $[\delta_t, \delta_r] = \alpha(t,r,\dot t, \dot r, w)\delta_w$, $k_8=k_9=0$ and $z = -w^2/u^2$ and $\Xi(t,r,z)$ is arbitrary so far.
 	 
 	 Since $k_8=k_9=0$ we have that $M=\tilde M$ and $N=\tilde N$ and the Ansatz $\Xi = z \xi$ transforms \eqref{eq:t-M} and \eqref{eq:r-N} into: 
 	 \begin{align}
 	 	\partial_t \xi + M z \partial_z \xi = 0\,,\quad \partial_r \xi + N z \partial_z \xi = 0\,,
 	 \end{align}
 	with integrability condition $\partial_r (M z \partial_z \xi) = \partial_t(N z \partial_z \xi)$. But, a quick calculation shows that this is equivalent to either $\partial_z\xi=0$ (which leads to a pseudo-Riemannian $L$ and is thus discarded), or
 	\begin{align}
 		\partial_r M - \partial_t N  = -2F = 0\,.
 	\end{align}
 	This equation is identically satisfied. On the given domain, therefore there exist a function $f(t,r)$ such that $M = \partial_t f$ and $N = \partial_r f$. Thus, the general solution of the remaining $\delta_t$ and $\delta_r$ equations is, in this case:
 	\begin{align}
 		\Xi(t,r,z) = z \xi(z e^{-f(t,r)})\,.
 	\end{align}
 	
 	\begin{conclusion}[Arbitrary $\Xi$ ($\boldsymbol{[\delta_t,\delta_r] = \alpha \delta_w $}):]\label{conc:XI1}
 		Given $A=B=C=D=E=F=0$ and $b=c=0$, then 
 		\begin{align}
 			L = u^2 \Xi = -w^2 \xi(q)\,,\quad q= - \frac{w^2}{u^2}e^{-f(t,r)}\,,
 		\end{align}
 		is of Berwald type, where $f = \int_C (M dt+Ndr)$ and $C$ is an arbitrary path from any point $t_0, r_0$ to $(t,r)$. Thus, we proved statement II.2.i.) of Theorem \ref{thm:main}. 
 	\end{conclusion}
 
 	\bigskip
 	
 	\item[3b.] 	We directly jump to the conclusion here, since for this case all derivations have been done already.
 	\begin{conclusion}[Arbitrary $\Xi$ ($\boldsymbol{[\delta_t,\delta_r] = 0$}):]\label{conc:XI2}
 	$L =  u^2 \Xi(t,r,z),\ z = v/u^2$, $D=E=F=0$, $b$ and $c$ are not both vanishing, then $[\delta_t,\delta_r] = 0$ identically. As already discussed in point $3$ of Lemma \ref{lem:trLie}, in this case there exist coordinates $\tilde t$ and $\tilde r$ such that $\delta_t \rightarrow \partial_{\tilde t}$ and $\delta_r \rightarrow \partial_{\tilde r}$. Thus, in these new coordinates, $L$ will be of the form
 	\begin{align}
 		\tilde L(\dot {\tilde  t}, \dot {\tilde r}, w) = u(\dot {\tilde  t}, \dot {\tilde r}, w)\Xi(z(\dot {\tilde  t}, \dot {\tilde r}, w))\,,
 	\end{align}
 	independent of $\tilde t$ and $\tilde r$. Thus, we proved statement II.2.ii.) of Theorem \ref{thm:main}.
 	\end{conclusion} 
\end{enumerate}

The identification of these four cases concludes our discussion of spherically symmetric the existence of Finsler manifolds of Berwald type, with connections for which $\delta_w$ is a nontrivial differential operator. We saw that for this case, non-Riemannian pseudo-Finsler structures of Berwald type exist and that they can be nicely classified into the four cases we listed.

\subsection{Case 2: Trivial $w$-equation}\label{sec:mainCase2}
The second major case we still need to discuss is the case when $\delta_w = 0$, i.e.\ if in addition to $k_{11}$ and $k_{12}$, we also have:
\begin{align}\label{eq:kw0}
	k_7=k_8=k_9=k_{10} = 0\,.
\end{align}
This immediately implies for the curvature components, see \eqref{a_i},
\begin{align}
	a_5 = a_6 = ... = a_{13} = 0,\quad a_{14}=1\,,
\end{align}
which means that \eqref{eq:curvtw} and \eqref{eq:curvrw} are also trivially satisfied. The constraints which can still be nontrivial are \eqref{eq:curvtr} and the iterated Lie bracket ones \eqref{eq:curvttr} and \eqref{eq:currtr}.

To find Berwald Finsler functions in this case, we again distinguish two subcases, namely if $[\delta_t, \delta_r]$ vanishes or not.

\bigskip
\subsubsection{$[\delta_t, \delta_r]=0$ }
This case is fairly simple. In addition to \eqref{eq:kw0}, as discussed previously in case 3. of Lemma \ref{lem:trLie} and Conclusion \ref{conc:XI2}, there exists a coordinate change $(t,r) \rightarrow (\tilde t,\tilde r)$ in which the $t$ and $r$ equations \eqref{eq:deltt} and \eqref{eq:deltr} reduce to $\partial_{\tilde t}L = \partial_{\tilde r} L =0$, meaning that any function $L$ with no dependence of $\tilde t$ or $\tilde r$ solves the Berwald conditions.

Note that, in these new coordinates, all coefficients $k_i$ need to be zero, since $\delta_w=0$, $\delta_{\tilde t} = \partial_{\tilde t}$  and $\delta_{\tilde r} = \partial_{\tilde r}$, and so the only nonvanishing connection coefficients are
\begin{align}
	N^\phi{}_\theta = \dot \phi \cot \theta\,,
	\quad N^\theta{}_\phi = \dot \phi \cos\theta \sin\theta\,,
	\quad N^\phi{}_\phi = \dot \theta \cot\theta\,.
\end{align}
These are precisely the Levi-Civita connection coefficients of the metric $w$ on the $2$-sphere. 

\begin{conclusion}[No $t$ and $r$ dependence:]\label{conc:notr}
	Any Finsler Lagrangian
	\begin{align}
	L = \dot t^2 \Xi(p,s)\,,\quad p = \frac{\dot r}{\dot t},\quad s = \frac{w}{\dot t}\,,
	\end{align}
	defines a Finsler manifold of Berwald type, for $\Xi$ being an arbitrary function of $p$ and $s$. Thus, we proved statement 1. of Theorem \ref{thm:main2}.
\end{conclusion}

\bigskip

\subsubsection{$[\delta_t, \delta_r]\neq0$ }
If $[\delta_t, \delta_r]\neq0$, then we still have to solve the curvature constraints \eqref{eq:curvtr}, \eqref{eq:curvttr} and \eqref{eq:currtr}. Since $a_5=0$ (and using the expressions of $A_5$ and $B_5$ in Appendix \ref{app:curv}), they are all of the form
\begin{align}\label{eq:w0tr}
	(X_1 \dot t + X_2 \dot r)\dot\partial_t L + (X_3 \dot t + X_4 \dot r)\dot \partial_r L = 0, \quad  X \in \{a,A,B\}\,.
\end{align}
Interpreting this system of three PDE's as algebraic system of equations in $\dot \partial_tL$ and  $\dot \partial_rL$, we can only find a nontrivial solution if all of these equations are proportional, that is:
\begin{align}
[\delta_t,[\delta_t, \delta_r]] \sim [\delta_t, \delta_r]\,,  [\delta_r,[\delta_t, \delta_r]] \sim [\delta_t, \delta_r].
\end{align}
Sorting this requirement according to the powers of  $\dot t$ and $\dot r$, we find the constraints:
\begin{align}\label{eq:aAB}
	A_3 a_1 &= A_1 a_3 & A_4 a_2 &= A_2 a_4 & A_3 a_2 - A_2 a_3 - A_1 a_4 + A_4 a_1 &= 0\nonumber\\
	B_3 a_1 &= B_1 a_3 & B_4 a_2 &= B_2 a_4 & B_3 a_2 - B_2 a_3 - B_1 a_4 + B_4 a_1 &= 0\,.
\end{align}

\bigskip
Assuming these conditions are satisfied, all the equations \eqref{eq:w0tr} are, actually, equivalent. Thus, it is sufficient to solve:
\begin{align}\label{eq:a0tr}
(a_1 \dot t + a_2 \dot r)\dot\partial_t L + (a_3 \dot t + a_4 \dot r)\dot \partial_r L = 0\,,
\end{align}
explicitly.

\bigskip

Introducing the new variables
\begin{align}
	p=\frac{\dot r}{\dot t}\,, \quad u = \dot t e^{I(t,r,p)}
\end{align}
with
\begin{align}\label{eq:defIp}
I(t,r,p) = \int \frac{(a_1+a_2 p)}{( a_2 p^2 - (a_4-a_1)p -a_3 )} dp\,
\end{align}
and using $(t,r,u, p, w)$ as independent variables in \eqref{eq:a0tr} we find that for being Berwald, $L$ must necessarily be of the form
\begin{align}
	L(t,r,\dot t, \dot r, w) = \Xi(t,r,u,w)\,.
\end{align} 

\bigskip
Furher, we insert the solution of the constraints equation into equations \eqref{eq:deltt} and \eqref{eq:deltr} to obtain
\begin{align}
	\delta_t L &=0 \quad \Leftrightarrow \quad \partial_t \Xi + K u \partial_u \Xi = 0\,,\label{eq:txi}\\
	\delta_r L &= 0 \quad \Leftrightarrow \quad \partial_r \Xi + T u \partial_u \Xi = 0\,,\label{eq:rxi}
\end{align}
with
\begin{align}
	K &= \partial_t I - (k_1 + k_2 p) + (k_1 p + k_2 p^2 - k_4 - k_6 p)\partial_p I\,,\\
	T &= \partial_r I - (k_2 + k_3 p) + (k_2 p + k_3 p^2 - k_6 - k_4 p)\partial_p I\,.
\end{align}
Since $\Xi$ and $\partial_u \Xi$ are independent of $p$, the equations \eqref{eq:txi} and \eqref{eq:rxi} imply $\partial_p K = 0 = \partial_p T$. Using \eqref{eq:defIp} one finds, by direct inspection, that these conditions are the same as \eqref{eq:aAB} and thus satisfied. Hence, we can conclude that 
\begin{align}
	K = K(t,r)\,,\quad T = T(t,r)\,.
\end{align}
Also, the integrability condition $\partial_r (K \partial_u \Xi) = \partial_t (T \partial_u \Xi)$ of \eqref{eq:txi} and \eqref{eq:rxi}, (which reduces to $\partial_r K  = \partial_tT$ by using  the $u$-derivative of \eqref{eq:txi} and \eqref{eq:rxi}), is identically satisfied by virtue of \eqref{eq:aAB}. The proof is straightforward using the definition of $K$ and $T$ and the definition of the coefficients $a_i$, $A_i$ and $B_i$ in terms of the connection coefficients as displayed in Appendix \ref{app:curv}.

The general solution of \eqref{eq:txi} and \eqref{eq:rxi} is now easily obtained as

\begin{align}
	\Xi(t,r,u,w) = \xi(u e^{-\varphi},w)\,,
\end{align}
where $\varphi$ is a solution of the system
\begin{align}
	\partial_t \varphi = K,\quad \partial_r \varphi = T\,.
\end{align}
Finally, using the 2-homogeneity of $L$, we can deduce the following.
\begin{conclusion}[Arbitrary $\Xi$: Case 3]\label{conc:xi3}
	Given $\delta_w$ is trivial and $[\delta_t,\delta_r]\neq 0$, such that $[\delta_t,[\delta_t, \delta_r]] \sim [\delta_t, \delta_r]\,,  [\delta_r,[\delta_t, \delta_r]] \sim [\delta_t, \delta_r]$, then 
	\begin{align}
		L =  w^2 \xi(q)\,,\quad q = \frac{\dot t e^{I-\varphi}}{w},\quad I =\int \frac{(a_1+a_2 p)}{( a_2 p^2 - (a_4-a_1)p -a_3 )} dp\,.\,,
	\end{align}
where $\xi$ is an arbitrary function of $q$, $\varphi = \int_C (K dt + T dr)$ and $C$ is an arbitrary path from any point $(t_0, r_0)$ to $(t,r)$, defines a Finsler manifold of Berwald type. Thus, we proved statement 2. of Theorem \ref{thm:main2}.
\end{conclusion}

\section{Conclusion}\label{sec:conc}
We classified all $4$-dimensional $SO(3)$-spherically symmetric, non-Riemannian pseudo-Finsler functions of Berwald type in the Theorems \ref{thm:main} and \ref{thm:main2}: there exist 6 such classes. Moreover we presented a further simplified version of the necessary and sufficient condition for a pseudo-Finsler space to be of Berwald type in Theorem \ref{thm:berw}, which we used to prove the classification. Our findings extend the classification of pseudo-Finsler spaces of Berwald type from homogeneous and isotropic symmetry \cite{Hohmann:2020mgs}.

We did not discuss the signature properties of the Finsler metric in all of our findings, i.e. they hold for positive definite Finsler spaces, as well as Lorentzian Finsler spacetimes. For future applications we mainly have the latter case in mind.

The next step in this research program is to use the forms of the Finsler Lagrangians \eqref{eq:1-L2}, \eqref{eq:1-L3}, \eqref{eq:1-L4},\eqref{eq:1-L5}, \eqref{eq:2-L1} and \eqref{eq:2-L2} as ansatz to:
\begin{itemize}
	\item construct spherically symmetric unicorn Finsler spacetimes;
	\item solve Finsler gravity equations to derive the gravitational field of Black Holes, and astrophysical compact objects such as ordinary and neutron stars described as kinetic gases.
\end{itemize}

A further future research direction is to consider the affine connection and the Finsler Lagrangian as independent variables and find solutions to the Palatini type Finsler affine gravity field equations suggested in \cite{Javaloyes:2021jqw}, in spatial spherical symmetry.

\appendix

\section{The curvature components and their derivatives}\label{app:curv}
Important necessary conditions for a Finsler manifold to be of Berwald type can be deduced from its curvature and derivatives thereof, see \eqref{eq:curvtr}-\eqref{eq:currtr}. Here, we display the curvature components in terms of the connection coefficients $k_{a}(t,r)$ and their derivatives.

The coefficients of the Lie brackets $\left[ \delta _{a},\delta _{b} \right]$, where $a\in \{t,r,\theta,\phi\}$ are defined by the curvature coefficients of the connection 
\begin{equation*}
	R^a{}_{bc}=\delta _{c}N^a{}_{b}-\delta_{b}N^a{}_{c} = \partial_c N^a{}_b - N^d{}_c \dot\partial_d N^a{}_b -  \partial_b N^a{}_c + N^d{}_b \dot\partial_d N^a{}_c\,.
\end{equation*}%
A direct computation gives
\begin{equation}
	\begin{array}{llll}
		R_{~tr}^{t}=a_{1}\dot{t}+a_{2}\dot{r} & R_{~tr}^{r}=a_{3}\dot{t}+a_{4}\dot{r} & R_{~tr}^{\theta }=a_{5}\dot{\theta} & R_{~tr}^{\varphi }=a_{5}\dot{\varphi} \\ 
		R_{~t\theta }^{t}=a_{6}\dot{\theta} & R_{~t\theta }^{r}=a_{7}\dot{\theta} & 
		R_{t\theta }^{\theta }=a_{8}\dot{t}+a_{9}\dot{r} & R_{t\theta }^{\varphi }=0
		\\ 
		R_{~t\varphi }^{t}=a_{6}\dot{\varphi}\sin ^{2}\theta  & R_{~t\varphi
		}^{r}=a_{7}\dot{\varphi}\sin ^{2}\theta  & R_{~t\varphi }^{\theta }=0 & 
		R_{t\varphi }^{\varphi }=a_{8}\dot{t}+a_{9}\dot{r} \\ 
		R_{~r\theta }^{t}=a_{10}\dot{\theta} & R_{~r\theta }^{r}=a_{11}\dot{\theta}
		& R_{~r\theta }^{\theta }=a_{12}\dot{t}+a_{13}\dot{r} & R_{~r\theta
		}^{\varphi }=0 \\ 
		R_{~r\varphi }^{t}=a_{10}\dot{\varphi}\sin ^{2}\theta  & R_{~r\varphi
		}^{r}=a_{11}\dot{\varphi}\sin ^{2}\theta  & R_{~r\varphi }^{\theta }=0 & 
		R_{~r\varphi }^{\varphi }=a_{12}\dot{t}+a_{13}\dot{r} \\ 
		R_{~\theta \varphi }^{t}=0 & R_{~\theta \varphi }^{r}=0 & R_{~\theta \varphi
		}^{\theta }=-a_{14}\dot{\varphi}\sin ^{2}\theta  & R_{~\theta \varphi
		}^{\varphi }=a_{14}\dot{\theta}\,,
	\end{array}
\end{equation}
where the coefficients $a_i$ are functions of $t$ and $r$ given by
\begin{align}\label{a_i}
	\begin{split}
	a_{1}  &= k_{1,r}-k_{2,t}+k_{3}k_{4}-k_{2}k_{6}\,,\\
	a_{2}  &= k_{2,r}-k_{3,t}+k_{2}^{2}+k_{3}k_{6}-k_{1}k_{3}-k_{2}k_{5}\,,  \\
	a_{3}  &= k_{4,r}-k_{6,t}+k_{1}k_{6}+k_{4}k_{5}-k_{2}k_{4}-k_{6}^{2}\,, \\
	a_{4}  &= k_{6,r}-k_{5,t}+k_{2}k_{6}-k_{3}k_{4}\,, \\
	a_{5}  &= k_{8,r}-k_{9,t}\,, \\
	a_{6}  &= -k_{7,t}+k_{7}k_{8}-k_{1}k_{7}-k_{2}k_{10}\,,\\
	a_{7}  &= -k_{10,t}+k_{8}k_{10}-k_{4}k_{7}-k_{6}k_{10}\,, \\
	a_{8}  &= -k_{8,t}+k_{1}k_{8}+k_{4}k_{9}-k_{8}^{2}\,,\\
	a_{9}  &= -k_{9,t}+k_{2}k_{8}+k_{6}k_{9}-k_{8}k_{9}  \,, \\
	a_{10} &= -k_{7,r}+k_{7}k_{9}-k_{2}k_{7}-k_{3}k_{10} \,,  \\
	a_{11} &= -k_{10,r}+k_{9}k_{10}-k_{6}k_{7}-k_{5}k_{10} \,,  \\
	a_{12} &= -k_{8,r}+k_{2}k_{8}+k_{6}k_{9}-k_{8}k_{9}\,,\\
	a_{13} &=-k_{9,r}+k_{3}k_{8}+k_{5}k_{9}-k_{9}^{2} \,, \\
	a_{14} &= 1+k_{7}k_{8}+k_{9}k_{10}  \,.
	\end{split}
\end{align}
Observe that in general $a_{5} = a_9 - a_{12}$ holds, as can be seen from the expressions above. 

Moreover, the Lie brackets between $\delta_w$ and $\delta_t$ or $\delta_r$ can be expressed in terms of these coefficients, as a direct calculation shows
\begin{align}
	[\delta_w, \delta_t] &= a_6 w \dot\partial_t  + a_7 w \dot\partial_r  + (a_8 \dot t + a_9 \dot r) \partial_w \,,\\
	[\delta_w, \delta_r] &= a_{10} w \dot\partial_t  + a_{11} w \dot\partial_r  + (a_{12} \dot t + a_{13} \dot r) \partial_w \,.
\end{align}

The double Lie brackets $\left[ \delta _{t},\left[ \delta _{t},\delta _{r}\right] \right]$ and $\left[ \delta _{r},\left[ \delta _{t},\delta _{r}\right] \right] $ are obtained by direct computation as:
\begin{eqnarray*}
	\left[ \delta _{t},\left[ \delta _{t},\delta _{r}\right] \right] &=&\left(
	A_{1}\dot{t}+A_{2}\dot{r}\right) \dot{\partial}_{t}+\left( A_{3}\dot{t}+A_{4}%
	\dot{r}\right) \dot{\partial}_{r}+A_{5}w{\partial}_{w}, \\
	\left[ \delta _{r},\left[ \delta _{t},\delta _{r}\right] \right] &=&\left(
	B_{1}\dot{t}+B_{2}\dot{r}\right) \dot{\partial}_{t}+\left( B_{3}\dot{t}+B_{4}%
	\dot{r}\right) \dot{\partial}_{r}+B_{5}w{\partial}_{w},
\end{eqnarray*}%
where:
\begin{align}\label{def_A_i_B_i}
	\begin{split}
	A_{1} &= \left( a_{1,t}+a_{3}k_{2}-a_{2}k_{4}\right) ,~\ \\
	A_{2} &=\left(a_{2,t}+a_{2}k_{1}-a_{1}k_{2}+a_{4}k_{2}-a_{2}k_{6}\right)\,,\\
	A_{3} &= \left( a_{3,t}-a_{3}k_{1}+a_{1}k_{4}-a_{4}k_{4}+a_{3}k_{6}\right)\,,\\
	A_{4} &= \left( a_{4,t}+a_{2}k_{4}-a_{3}k_{2}\right)\,,\\
	A_{5} &= a_{5,t}\,, \\
	B_{1} &= \left( a_{1,r}+a_{3}k_{3}-a_{2}k_{6}\right)\,, \\
	B_{2} &= \left(a_{2,r}+a_{2}k_{2}-a_{1}k_{3}+a_{4}k_{3}-a_{2}k_{5}\right)\,, \\
	B_{3} &= \left( a_{3,r}-a_{3}k_{2}+a_{1}k_{6}-a_{4}k_{6}+a_{3}k_{5}\right)\,,\\
	B_{4} &= \left( a_{4,r}+a_{2}k_{6}-a_{3}k_{3}\right)\,,\\
	B_{5} &=a_{5,r}\,.
	\end{split}
\end{align}

	C.P. was funded by the Deutsche Forschungsgemeinschaft (DFG, German Research Foundation) - Project Number 420243324.
	S.Ch. was funded by the Transilvania Fellowships for Young Researchers 2021 grant. 
	This article is based upon work from COST Actions: CA21136 (Addressing observational tensions in cosmology with systematics and fundamental physics - CosmoVerse) and CA18108 (Quantum Gravity Phenomenology in the Multi-Messenger Approach), supported by COST (European Cooperation in Science and Technology).

\bibliographystyle{utphys}
\bibliography{SBS}

\providecommand{\href}[2]{#2}\begingroup\raggedright\begin{thebibliography}{10}

\bibitem{Wald}
R.~M. Wald, {\em {General Relativity}}.
\newblock The University of Chicago Press,, 1984.

\bibitem{Saridakis:2021vue}
E.~N. Saridakis, R.~Lazkoz, V.~Salzano, P.~Vargas~Moniz, S.~Capozziello,
  J.~Beltr\'an~Jim\'enez, M.~De~Laurentis, and G.~J. Olmo, eds.,
  \href{http://dx.doi.org/10.1007/978-3-030-83715-0}{{\em {Modified Gravity and
  Cosmology}}}.
\newblock Springer, 2021.

\bibitem{Addazi:2021xuf}
A.~Addazi {\em et al.}, ``{Quantum gravity phenomenology at the dawn of the
  multi-messenger era\textemdash{}A review},''
  \href{http://dx.doi.org/10.1016/j.ppnp.2022.103948}{{\em Prog. Part. Nucl.
  Phys.} {\bf 125} (2022)  103948}, \href{http://arxiv.org/abs/2111.05659}{{\tt
  arXiv:2111.05659 [hep-ph]}}.

\bibitem{Hohmann:2019fvf}
M.~Hohmann, ``{Metric-affine Geometries With Spherical Symmetry},''
  \href{http://dx.doi.org/10.3390/sym12030453}{{\em Symmetry} {\bf 12} (2020)
  no.~3, 453},
\href{http://arxiv.org/abs/1912.12906}{{\tt arXiv:1912.12906 [math-ph]}}.

\bibitem{Berwald1926}
L.~Berwald, \href{http://dx.doi.org/10.1007/BF01283825}{``Untersuchung der
  {K}r{\"u}mmung allgemeiner metrischer {R}{\"a}ume auf {G}rund des in ihnen
  herrschenden {P}arallelismus,''{\em Mathematische Zeitschrift} {\bf 25} (Dec,
  1926)  40--73}. \url{https://doi.org/10.1007/BF01283825}.

\bibitem{Finsler}
P.~Finsler, {\em {\"Uber Kurven und Fl\"achen in allgemeinen R\"aumen}}.
\newblock PhD thesis, {Georg-August Universit\"at zu G\"ottingen}, 1918.

\bibitem{Bao}
D.~Bao, S.-S. Chern, and Z.~Shen, {\em {An Introduction to Finsler-Riemann
  Geometry}}.
\newblock Springer, New York, 2000.

\bibitem{Miron}
R.~Miron and I.~Bucataru, {\em {Finsler Lagrange Geometry}}.
\newblock Editura Academiei Romane, Bucharest, 2007.

\bibitem{Voicu2018}
N.~Voicu, ``Conformal maps between pseudo-finsler spaces,''
  \href{http://dx.doi.org/10.1142/S0219887818500032}{{\em International Journal
  of Geometric Methods in Modern Physics} {\bf 15} (2018) no.~01, 1850003},
  \href{http://arxiv.org/abs/https://doi.org/10.1142/S0219887818500032}{{\tt
  https://doi.org/10.1142/S0219887818500032}}.
  \url{https://doi.org/10.1142/S0219887818500032}.

\bibitem{Fuster:2020upk}
A.~Fuster, S.~Heefer, C.~Pfeifer, and N.~Voicu, ``{On the non metrizability of
  Berwald Finsler spacetimes},''
  \href{http://dx.doi.org/10.3390/universe6050064}{{\em Universe} {\bf 6}
  (2020) no.~5, 64}, \href{http://arxiv.org/abs/2003.02300}{{\tt
  arXiv:2003.02300 [math.DG]}}.

\bibitem{Javaloyes:2022hph}
M.~A. Javaloyes, E.~Pend\'as-Recondo, and M.~S\'anchez, ``{An account on links
  between Finsler and Lorentz Geometries for Riemannian Geometers},''
  \href{http://arxiv.org/abs/2203.13391}{{\tt arXiv:2203.13391 [math.DG]}}.

\bibitem{Pfeifer:2011tk}
C.~Pfeifer and M.~N.~R. Wohlfarth, ``{Causal structure and electrodynamics on
  Finsler spacetimes},'' {\em Phys.Rev.} {\bf D84} (2011)  044039,
  \href{http://arxiv.org/abs/1104.1079}{{\tt arXiv:1104.1079 [gr-qc]}}.

\bibitem{Minguzzi:2014fxa}
E.~Minguzzi, ``{The connections of pseudo-Finsler spaces},''
  \href{http://dx.doi.org/10.1142/S0219887814600251,
  10.1142/S0219887815920012}{{\em Int. J. Geom. Meth. Mod. Phys.} {\bf 11}
  (2014) no.~07, 1460025}, \href{http://arxiv.org/abs/1405.0645}{{\tt
  arXiv:1405.0645 [math-ph]}}.
[Erratum: Int. J. Geom. Meth. Mod. Phys.12,no.7,1592001(2015)].

\bibitem{Gomez-Lobo:2016qik}
A.~Garc\'\i{}a-Parrado G\'omez-Lobo and E.~Minguzzi, ``{Pseudo-{F}insler spaces
  modeled on a pseudo-{M}inkowski space},''
  \href{http://dx.doi.org/10.1016/S0034-4877(18)30069-7}{{\em Rept. Math.
  Phys.} {\bf 82} (2018)  29--42}, \href{http://arxiv.org/abs/1612.00829}{{\tt
  arXiv:1612.00829 [math.DG]}}.

\bibitem{Minguzzi2016}
E.~Minguzzi, ``Special coordinate systems in pseudo-{F}insler geometry and the
  equivalence principle,''
  \href{http://dx.doi.org/https://doi.org/10.1016/j.geomphys.2016.12.013}{{\em
  Journal of Geometry and Physics} {\bf 114} (2017)  336 -- 347}.

\bibitem{Javaloyes:2018lex}
M.~A. Javaloyes and M.~Sánchez, ``{On the definition and examples of cones and
  Finsler spacetimes},''
\href{http://arxiv.org/abs/1805.06978}{{\tt arXiv:1805.06978 [math.DG]}}.

\bibitem{Minculete_2021}
N.~Minculete, C.~Pfeifer, and N.~Voicu, ``Inequalities from {L}orentz-{F}insler
  norms,'' \href{http://dx.doi.org/10.7153/mia-2021-24-26}{{\em Mathematical
  Inequalities and Applications} (2021) no.~2, 373--398}.
  \url{https://doi.org//10.7153/mia-2021-24-26}.

\bibitem{Tavakol1986}
R.~K. Tavakol and N.~Van~den Bergh, ``{Viability criteria for the theories of
  gravity and finsler spaces},''
  \href{http://dx.doi.org/10.1007/BF00770205}{{\em General Relativity and
  Gravitation} {\bf 18} (1986) no.~8, 849 -- 859}.

\bibitem{Voicu:2009wi}
N.~Voicu, ``{New considerations on Hilbert action and Einstein equations in
  anisotropic spaces},'' \href{http://dx.doi.org/10.1063/1.3506066}{{\em AIP
  Conf. Proc.} {\bf 1283} (2010)  249--257},
\href{http://arxiv.org/abs/0911.5034}{{\tt arXiv:0911.5034 [gr-qc]}}.

\bibitem{Pfeifer:2011xi}
C.~Pfeifer and M.~N.~R. Wohlfarth, ``{Finsler geometric extension of Einstein
  gravity},'' \href{http://dx.doi.org/10.1103/PhysRevD.85.064009}{{\em Phys.
  Rev.} {\bf D85} (2012)  064009},
\href{http://arxiv.org/abs/1112.5641}{{\tt arXiv:1112.5641 [gr-qc]}}.

\bibitem{Lammerzahl:2018lhw}
C.~Lammerzahl and V.~Perlick, ``{Finsler geometry as a model for relativistic
  gravity},''
\newblock 2018.
\newblock \href{http://arxiv.org/abs/1802.10043}{{\tt arXiv:1802.10043
  [gr-qc]}}.
\newblock
\url{https://inspirehep.net/record/1657800/files/1802.10043.pdf}.
\newblock

\bibitem{Pfeifer:2019}
C.~Pfeifer, ``{Finsler spacetime geometry in Physics},''
  \href{http://dx.doi.org/10.1142/S0219887819410044}{{\em Int. J. Geom. Meth.
  Mod. Phys.} {\bf 16} (2019) no.~supp02, 1941004},
  \href{http://arxiv.org/abs/1903.10185}{{\tt arXiv:1903.10185 [gr-qc]}}.

\bibitem{Hohmann_2019}
M.~Hohmann, C.~Pfeifer, and N.~Voicu, ``{Finsler gravity action from
  variational completion},''
  \href{http://dx.doi.org/10.1103/PhysRevD.100.064035}{{\em Phys. Rev. D} {\bf
  100} (2019) no.~6, 064035}, \href{http://arxiv.org/abs/1812.11161}{{\tt
  arXiv:1812.11161 [gr-qc]}}.

\bibitem{Hohmann:2019sni}
M.~Hohmann, C.~Pfeifer, and N.~Voicu, ``{Relativistic kinetic gases as direct
  sources of gravity},''
  \href{http://dx.doi.org/10.1103/PhysRevD.101.024062}{{\em Phys. Rev.} {\bf
  D101} (2020) no.~2, 024062},
\href{http://arxiv.org/abs/1910.14044}{{\tt arXiv:1910.14044 [gr-qc]}}.

\bibitem{Lobo:2020qoa}
I.~P. Lobo and C.~Pfeifer, ``{Reaching the Planck scale with muon lifetime
  measurements},'' \href{http://dx.doi.org/10.1103/PhysRevD.103.106025}{{\em
  Phys. Rev. D} {\bf 103} (2021) no.~10, 106025},
  \href{http://arxiv.org/abs/2011.10069}{{\tt arXiv:2011.10069 [hep-ph]}}.

\bibitem{Kapsabelis:2022plf}
E.~Kapsabelis, P.~G. Kevrekidis, P.~C. Stavrinos, and A.~Triantafyllopoulos,
  ``{Schwarzschild-Finsler-Randers spacetime: Dynamical analysis, Geodesics and
  Deflection Angle},'' \href{http://arxiv.org/abs/2208.05063}{{\tt
  arXiv:2208.05063 [gr-qc]}}.

\bibitem{Carvalho:2022sdz}
P.~Carvalho, C.~Landri, R.~Mistry, and A.~Pinzul, ``{Multimetric Finsler
  Geometry},'' \href{http://arxiv.org/abs/2208.03800}{{\tt arXiv:2208.03800
  [math-ph]}}.

\bibitem{Garcia-Parrado:2022ith}
A.~Garcia-Parrado and E.~Minguzzi, ``{An anisotropic gravity theory},''
  \href{http://arxiv.org/abs/2206.09653}{{\tt arXiv:2206.09653 [gr-qc]}}.

\bibitem{Zhu:2022blp}
J.~Zhu and B.-Q. Ma, ``{Lorentz-violation-induced arrival time delay of
  astroparticles in Finsler spacetime},''
  \href{http://dx.doi.org/10.1103/PhysRevD.105.124069}{{\em Phys. Rev. D} {\bf
  105} (2022) no.~12, 124069}, \href{http://arxiv.org/abs/2206.07616}{{\tt
  arXiv:2206.07616 [gr-qc]}}.

\bibitem{Aazami:2022bib}
A.~B. Aazami, M.~A. Javaloyes, and M.~C. Werner, ``{Finsler pp-waves and the
  Penrose Limit},'' \href{http://arxiv.org/abs/2205.01162}{{\tt
  arXiv:2205.01162 [math.DG]}}.

\bibitem{Javaloyes:2022fmp}
M.~A. Javaloyes, M.~S\'anchez, and F.~F. Villase\~nor, ``{On the Significance
  of the Stress\textendash{}Energy Tensor in Finsler Spacetimes},''
  \href{http://dx.doi.org/10.3390/universe8020093}{{\em Universe} {\bf 8}
  (2022) no.~2, 93}, \href{http://arxiv.org/abs/2202.10801}{{\tt
  arXiv:2202.10801 [gr-qc]}}.

\bibitem{Hohmann:2020mgs}
M.~Hohmann, C.~Pfeifer, and N.~Voicu, ``{Cosmological Finsler Spacetimes},''
  \href{http://dx.doi.org/10.3390/universe6050065}{{\em Universe} {\bf 6}
  (2020) no.~5, 65}, \href{http://arxiv.org/abs/2003.02299}{{\tt
  arXiv:2003.02299 [gr-qc]}}.

\bibitem{Elgendi:2021-1}
S.~G. Elgendi, ``{On the classification of Landsberg spherically symmetric
  Finsler metrics},'' \href{http://dx.doi.org/10.1142/S0219887821502327}{{\em
  International Journal of Geometric Methods in Modern Physics} {\bf 18} (2021)
  no.~14, 2150232}, \href{http://arxiv.org/abs/2110.07252}{{\tt
  arXiv:2110.07252 [math.DG]}}.

\bibitem{Hohmann:2018rpp}
M.~Hohmann, C.~Pfeifer, and N.~Voicu, ``{Finsler gravity action from
  variational completion},''
  \href{http://dx.doi.org/10.1103/PhysRevD.100.064035}{{\em Phys. Rev.} {\bf
  D100} (2019) no.~6, 064035},
\href{http://arxiv.org/abs/1812.11161}{{\tt arXiv:1812.11161 [gr-qc]}}.

\bibitem{Hohmann:2020yia}
M.~Hohmann, C.~Pfeifer, and N.~Voicu, ``{The kinetic gas universe},''
  \href{http://dx.doi.org/10.1140/epjc/s10052-020-8391-y}{{\em Eur. Phys. J. C}
  {\bf 80} (2020) no.~9, 809}, \href{http://arxiv.org/abs/2005.13561}{{\tt
  arXiv:2005.13561 [gr-qc]}}.

\bibitem{ELGENDI2021103918}
S.~Elgendi, ``Solutions for the landsberg unicorn problem in finsler
  geometry,''
  \href{http://dx.doi.org/https://doi.org/10.1016/j.geomphys.2020.103918}{{\em
  Journal of Geometry and Physics} {\bf 159} (2021)  103918}.
  \url{https://www.sciencedirect.com/science/article/pii/S0393044020302151}.

\bibitem{Muzsnay2006TheEP}
Z.~Muzsnay, ``The euler-lagrange pde and finsler metrizability,'' {\em HOUSTON
  JOURNAL OF MATHEMATICS} {\bf 32} (2019) no.~1, .

\bibitem{Bucataru_2011}
I.~Bucataru and Z.~Muzsnay, ``Projective metrizability and formal
  integrability,'' \href{http://dx.doi.org/10.3842/sigma.2011.114}{{\em
  Symmetry, Integrability and Geometry: Methods and Applications} (2011)  }.
  \url{https://doi.org/10.3842%2Fsigma.2011.114}.

\bibitem{Beem}
J.~K. Beem, ``Indefinite {F}insler spaces and timelike spaces,'' {\em Can. J.
  Math.} {\bf 22} (1970)  1035.

\bibitem{Hohmann:2021zbt}
M.~Hohmann, C.~Pfeifer, and N.~Voicu, ``{Mathematical foundations for field
  theories on Finsler spacetimes},''
  \href{http://dx.doi.org/10.1063/5.0065944}{{\em J. Math. Phys.} {\bf 63}
  (2022) no.~3, 032503}, \href{http://arxiv.org/abs/2106.14965}{{\tt
  arXiv:2106.14965 [math-ph]}}.

\bibitem{Pfeifer:2019tyy}
C.~Pfeifer, S.~Heefer, and A.~Fuster, ``{Identifying Berwald Finsler
  geometries},'' \href{http://dx.doi.org/10.1016/j.difgeo.2021.101817}{{\em
  Differ. Geom. Appl.} {\bf 79} (2021)  101817},
  \href{http://arxiv.org/abs/1909.05284}{{\tt arXiv:1909.05284 [math.DG]}}.

\bibitem{Bejancu}
H.~R.~F. Aurel~Bejancu, {\em {Geometry of Pseudo-Finsler Submanifolds}}.
\newblock Springer, Dodrecht, 2000.

\bibitem{Szilasi2011}
J.~Szilasi, R.~L. Lovas, and D.~C. Kertesz, ``Several ways to {B}erwald
  manifolds - and some steps beyond,'' {\em Extracta Math.} {\bf 26} (2011)
  89--130, \href{http://arxiv.org/abs/1106.2223}{{\tt arXiv:1106.2223
  [math.DG]}}.

\bibitem{Szabo}
Z.~Szab\'o, ``Positive definite berwald spaces,'' {\em Tensor, New Series} {\bf
  35} (1981)  25--39.

\bibitem{Krupka-Brajercik}
D.~Krupka and J.~Brajerčík, ``Schwarzschild spacetimes: Topology,''
  \href{http://dx.doi.org/10.3390/axioms11120693}{{\em Axioms} {\bf 11} (2022)
  no.~12, }. \url{https://www.mdpi.com/2075-1680/11/12/693}.

\bibitem{Javaloyes:2021jqw}
M.~A. Javaloyes, M.~S\'anchez, and F.~F. Villase\~nor, ``{The
  Einstein-Hilbert-Palatini formalism in Pseudo-Finsler Geometry},''
  \href{http://arxiv.org/abs/2108.03197}{{\tt arXiv:2108.03197 [math.DG]}}.

\end{thebibliography}\endgroup

\end{document}